\documentclass[11pt]{article}
\usepackage[T1]{fontenc}

\usepackage{authblk}
\usepackage{fullpage}
\usepackage{algorithmic}
\usepackage{algorithm}
\usepackage{epsfig}
\usepackage{graphicx}
\usepackage{latexsym}
\usepackage{amsmath}
\usepackage{amsfonts}
\usepackage{multirow}
\usepackage{amssymb}
\usepackage[sort,numbers]{natbib}
 \usepackage{tikz}

\usetikzlibrary{shapes,arrows,shadows}
\usepackage{amsmath,bm,times}

\usepackage{mathrsfs}
\usepackage{pifont}
\usepackage{yhmath}
\usepackage{bbm}
\usepackage{amsthm}
\usepackage{booktabs}
\usepackage{stmaryrd}
\usepackage{wasysym}
\usepackage{bbm}
\usepackage{array}
\usepackage{enumerate}
\usepackage{dsfont}

\numberwithin{equation}{section}
\numberwithin{figure}{section}

\newtheorem{theorem}{Theorem}[section]
\newtheorem{lemma}[theorem]{Lemma}

\newtheorem{corollary}[theorem]{Corollary}

\newtheorem{remark}[theorem]{Remark}

\theoremstyle{remark}

\makeatletter
\newcommand\figcaption{\def\@captype{figure}\caption}
\newcommand\tabcaption{\def\@captype{table}\caption}
\makeatother

\DeclareMathAlphabet{\mathpzc}{OT1}{pzc}{m}{it}

\begin{document}
\newcounter{my}
\newenvironment{mylabel}
{
\begin{list}{(\roman{my})}{
\setlength{\parsep}{-1mm}
\setlength{\labelwidth}{8mm}
\usecounter{my}}
}{\end{list}}

\newcounter{my2}
\newenvironment{mylabel2}
{
\begin{list}{(\alph{my2})}{
\setlength{\parsep}{-0mm} \setlength{\labelwidth}{8mm}
\setlength{\leftmargin}{3mm}
\usecounter{my2}}
}{\end{list}}

\newcounter{my3}
\newenvironment{mylabel3}
{
\begin{list}{(\alph{my3})}{
\setlength{\parsep}{-1mm}
\setlength{\labelwidth}{8mm}
\setlength{\leftmargin}{10mm}
\usecounter{my3}}
}{\end{list}}

\title{\bf The behavior of  renormalization and related  observables \thanks{\noindent{\bf 2000 Mathematics Subject Classification}\quad 60K35; 60H10}}
\author[a]{Cui Kaiyuan
}
	\author[a]{Gong Fuzhou}
	\affil[a]{Institute of Applied Mathematics, Academy of Mathematics and Systems Science, Chinese Academy of Sciences, Beijing 100080, China,}
	\renewcommand*{\Affilfont}{\small\it}
	\renewcommand\Authands{and}
	\date{}
\maketitle

\vspace{-5em}

\begin{center}\large

\end{center}


\begin{abstract}
In this paper, we  introduce new reference observables to establish a scaling formula in the renormalization group equation. Using the transfer matrix method, we calculate the two point  observables of the one dimensional Ising model without an external field under general boundary conditions. The 
results indicate that the two point observables exhibit exponential decay as the distance between these two sites tends to infinity, except at the critical point.
Corresponding to the  renormalization procedure underlying  the correlation function, we establish a similar procedure for new observables, which aligning with findings in physics. Additionally, from the dynamic point of view, we construct a random system using the stochastic quantization method. We calculate the new observables of this random system under the initial distribution that satisfies Dobrushin Lanford Ruelle(DLR) equations. Furthermore, we formulate a new renormalization scaling equation with respect to the two point observables.  Finally, these results can be extended to a more general case of finite point  observables, and demonstrating independence from the choice of system parameters.\\

\noindent{\bf Keywords:}~Observables; Ising model; Stochastic quantization; Renormalization.
\end{abstract}
	\section{Introduction}\label{sec:1}
Phase transitions and critical phenomena are crucial topics in statistical physics.
The focus upon phase transitions produced a set of remarkably important contributions to physics via the synthesis and construction of the modern renormalization group technique; see \cite{Kadanoff2013document}. 
In 1971, Wilson formulated the comprehensive theory known as \emph{the renormalization group(RG) theory}, combining Kadanoff's block scaling analysis with the idea of renormalization. Shortly thereafter, he was awarded the Nobel Prize for his contributions to understanding critical phenomena. Nowadays, \emph{the renormalization group theory} has been a powerful instrument for analyzing different strongly coupled system and has matured as
the leading computational technique for determining the properties of
systems exhibiting self-similarity under rescaling; see	\cite{Apenko2012document,Robledo2000document}.

Actually, in the classic procedure of  \emph{Real Space Renormalization Group(RSRG)}, the pivotal idea  is  the ``coarse-graining" procedure formulated by Kaddanoff in 1966. It states that, no matter how many times the blocking transformation is iterated, the dominant interactions will be short-ranged; see \cite{Cardy1996book}. The blocking transformation maps the  original ``spins''  to a sequence of new ``block spins'' but leaves the structure of probability distributions or the partition function invariant; see Figure 1 in \cite{Cui2022document} for details. Specifically, during this procedure, the partition function remains unchanged and correlation functions satisfies a scaling formula when the real positive dilatation of spatial points diverges; see \cite{Zinn2014book}. However, it is very hard, even impossible in infinite dimensional, to calculate partition function. Hence, it is essential to identify a suitable observable that not only characterizes the renormalization procedure physically but can also be estimated or calculated mathematically. This constitutes the main goal of this paper.

To gain a deeper understanding of the mechanism
underlying the renormalization group procedure, researchers have explored dynamic perspectives. Boltzmann generalized the concept of entropy to the non-equilibrium situations through the quantity $H$; see
\cite{Balian2005document}. The analogous nonequilibrium ideas can
be associated with Wilsonian renormalization. Zamolodchikov's $c$ -theorem in two dimensional quantum field theory provided an initial  successful pattern; see \cite{Zamolodchikov1986document}. Much work has been done in this direction; see
\cite{Apenko2012document,Cardy1988document,Forte1998document,Amariti2011document}
and new physical quantities such as relative entropy and entanglement entropy, are introduced; see
\cite{Gaite1996document,Gaite1998document,Gaite2000document1,Gaite2000document2,Casini2017document}. 
Generally, the RG procedure can be regard as a system of autonomous ordinary differential equation on the space of coupling constants. This flow typically possesses attractive
fixed points or surfaces corresponding to the physics at large
length scales; see \cite{beny2015document}.
Investigators have made many efforts to study the RG flow, such as  
adding a noise to the RG equations; see \cite{Gaite2004document} 
or establishing the connection between gradient flow and RG flow; see \cite{Narayanan2006document,Luscher2010document,Abe2018document}.
In the functional renormalization group, Carosso denote the time parameter $t$ not as a physical time, but rather an $RSRG$ time, they establish the renormalization group through a stochastic Langevin dynamic and formulate a new Monte Carlo  renormalization group method; see \cite{Carosso2020document}.

However, we approach the evolution differently, focusing not on investigating dynamics in the coupling parameter space but on the renormalization group procedure itself. 
Drawing inspiration from the concept of stochastic quantization, as proposed by Parisi and Wu, we  introduce a novel perspective to reconstruct the renormalization procedure; see \cite{Cui2022document} for more details. Within this framework, the scaling properties of new observables are analyzed. Ultimately, we establish the observable renormalization principle based on an analogous scaling formula. 

As we have stated in paper \cite{Cui2022document}, considering the Ising-like Hamiltonian without an external field 
\begin{align}\label{1.1}
	\mathcal{H} =-K\sum_{i} \phi(x_{i})\phi(x_{i+1})+\frac{\gamma}{2}\sum_{i} \phi^{2}(x_{i}),
\end{align}
where $x_{i}\in \mathbb{Z}$ is the spatial site and $\gamma$ is constant offset such that the Hamiltonian is positive definite; see \cite{Li2018document}. 
As we discussed in \cite{Cui2022document}, inspired by the idea of stochastic quantization in \cite{Damgaard1987Huffel}, under the framework of Euler’s method in time discretisation \cite{Gyongydocument1999},  we obtain a new version of renormalization  procedure formally 
\begin{align}\label{eq:1.7}
	\phi(x_{i},t+1)=sgn\{K\tilde{\Delta}\phi(x_{i},t)+(2K-\gamma+1)\phi(x_{i},t)+\xi(x_{i},t)\}.
\end{align} 
In principle, distinct coarse-graining procedures result in  different renormalization group schemes; see  \cite{Cardy1996book,Morandi2004document}. Hence, it is interesting to investigate the properties of equation~\eqref{eq:1.7}, and it may bring some amazing observations with the help of new tools. Denote that $\phi(x,t)=(\cdots,\phi(x_{-1},t),\phi(x_{0},t),\phi(x_{1},t),\cdots)$, and $\Pi_{\{i\}}x=x_{i},\Pi_{\{i\}}\phi(x,t)=\phi(x_{i},t).$ Then equation~\eqref{eq:1.7}~can be written as
\begin{align}\label{eq:1.8}
	\phi(x,t+1)=\mathcal{P}(\phi(x,t),\xi(x,t)),
\end{align}
where
\begin{align*}
	\Pi_{\{i\}}\mathcal{P}(\phi(x,t),\xi(x,t))=sgn\{K\tilde{\Delta}\phi(x_{i},t)+(2K-\gamma+1)\phi(x_{i},t)+\xi(x_{i},t)\},
\end{align*}
and $\{\xi(x_{i},t)\}_{i\in\mathbb{Z},t\in\mathbb{Z}^{+}}$ are independent Gaussian random variables. The stochastic dynamic~\eqref{eq:1.8}~ can be seen as a Markov chain on state space $E=\{+ 1,-1\}^{\mathbb{Z}}$, which is time-homogeneous if $K$ is a constant. 

To sum up, the objective of this paper is twofold. Firstly, it aims to  reconstruct the transition function in renormalization procedure through a stochastic difference equation.
Secondly, it seeks to identify suitable observables and propose a new principle that can effectively characterize the renormalization procedure. These efforts lay the foundation for further study. Our paper is organized as follows.
In Section 2, for reader's convenience, we list some elementary results in statistical mechanics.
In Section 3, we calculate the two point  observables of Ising model with general boundary condition. The results indicate that the two point  observables exhibit comment properties, such as exponential decay, shared with two point correlation functions. 
In Section 4, we propose a new renormalization procedure in the probabilistic terms based on the observables in Section 2 and  verify our idea 
in the case of one dimensional Ising model. In Section 5, from the stochastic point of view, we consider the convergence of two point observables along the evolution of stochastic dynamic~\eqref{eq:1.8}~.
In Section 6, the calculation of two point observables are extended to general cases.
\section{Preliminaries}\label{sec:2}

Let's consider the one dimensional Ising model with nearest neighbor interaction,  the spins are defined on one dimensional lattice $x_{i}\in \mathds{Z}$. Hamiltonian $\mathcal{H}$ with a constant external field $H$ has the following form	
\begin{align}
	\mathcal{H}=-J\sum_{i=1}^{N}\sigma_{x_{i}}\sigma_{x_{i+1}}-H\sum_{i=1}^{N}\sigma_{x_{i}},
\end{align}
where $J>0$ is the ferromagnetic coupling constant, and $\sigma_{x_{N+1}}=\sigma_{x_{1}}$ which means periodic boundary conditions.
Then the canonical ensemble distribution corresponding to a system of $N$ spins in thermal equilibrium is	$\mathds{P}_{N}=\frac{1}{Z_{N}}e^{-\beta\mathcal{H}}$,  where $Z_{N}$ is the \emph{partition function}, which can be written as
$$Z_{N}=\sum_{\{\sigma_{x_{i}}=\pm 1\}}exp\{\beta J \sum_{i=1}^{N}\sigma_{x_{i}}\sigma_{x_{i+1}}+\beta H\sum_{i=1}^{N}\sigma_{x_{i}}\} $$
Traditionally, the parameter $\beta$ will be incorporated into the coupling constant $J$ and external field $H$, that is to say, $K=\beta J$ and $h=\beta H$. By means of the \emph{tranfer matrix}, one can get some fundamental results of the partition function. Furthermore, the \emph{Real Space Renormalization Group (RSRG)} method is an important approach to study critical phenomena, and during the renormalization group procedure, the partition function remains unchanged.  In the case of one dimensional Ising chain, the "coarse-graining" procedure give us the renormalization group transformation (\emph{RGT}); see \cite{Morandi2004document} for more details. For reader's convenience, we will list some explicit results bellow: 
\begin{enumerate}
	\item $Z_{N}$ can be expressed as $Z_{N}=\lambda^{N}_{+}+\lambda^{N}_{-}$, where $\lambda_{\pm}=e^{K}[\cosh(h)\pm\sqrt{\sinh^{2}(h)+e^{-4K}}]$;
	\item Under thermodynamic limit, the two point correlation function  $\mathcal{G}(x_{i},x_{j})=\mathcal{G}_{i,j}$ can be calculated		\begin{align}\label{2.10}
		\mathcal{G}_{i,j}=sin^{2}(2\phi) \left (\frac{\lambda_{-}}{\lambda_{+}}\right)^{|x_{i}-x_{j}|},
	\end{align}	
	where $\tan(2\phi)=\frac{e^{-2K}}{\sinh(h)}$;
	\item 
	The transformations: $K\rightarrow K^{'}$, and $h\rightarrow h^{'}$are called \emph{RG} transformation(\emph{RGT}), where $K^{'}$ and $h^{'}$ are couple constants of new "coarse-graining" Hamiltonian $\mathcal{H^{'}}$. Particularly, if $h=0$, $K^{'}$ and $h^{'}$ are determined by		\begin{align}\label{2.11}
		K^{'}&=-\frac{1}{2}\log\left(\frac{2}{e^{K}+e^{-K}}\right),
		h^{'}=0.
	\end{align}
\end{enumerate}	
In the field of thermodynamics, the behavior of materials  near their critical points of phase transitions was a quite fashionable subject in the early 1960s; see \cite{Kadanoffdocument2015}. Particularly, the correlation functions play a central role in studying thermodynamic systems, and the $n$-point correlation functions can be defined as
$$\mathcal{G}(x_{1},x_{2},\cdots,x_{n})=\lim_{N \to \infty} \langle \sigma_{x_{1}}\sigma_{x_{2}}\cdots\sigma_{x_{n}}\rangle_{\mathds{P}_{N}},$$
where $ \langle\cdot\rangle_{\mathds{P}_{N}}$ means the expectation with respect to measure $\mathds{P}_{N}$, and $N \to \infty$ is the thermodynamic limit.	

Let us consider some general observables from the  probabilistic point of view.
For any $f \in L^{2}(\mathds{R})$, let $\mathcal{S}^{f}$ denote the observable induced by $f$ under the thermodynamic limit, and define
\begin{align}\label{2.13}
	\mathcal{S}^{f} := \hat{\mathcal{S}}^{f} - \tilde{\mathcal{S}}^{f}:=\lim_{N \to \infty}\hat{\mathcal{S}}_{N}^{f}-\lim_{N \to \infty}\tilde{\mathcal{S}}_{N}^{f},
\end{align}
where $$\hat{\mathcal{S}}_{N}^{f}=\langle f^{2}(\sigma_{x_{1}})\log(f^{2}(\sigma_{x_{1}}))\rangle_{\mathds{P}_{N}}, \tilde{\mathcal{S}}_{N}^{f}=\langle f^{2}(\sigma_{x_{1}})\rangle_{\mathds{P}_{N}}\log(\langle f^{2}(\sigma_{x_{1}})\rangle_{\mathds{P}_{N}}).$$
The $n$-point observables can be defined similarly. A natural intuition is to establish an analogous scaling formula as~\eqref{2.12}~ below based on $n$-point observables. Before that, the elementary properties of  two point observables need to be investigated, that is our main goal in next section. 	

\section{The two point observables of one dimensional Ising model}\label{sec:3}
This section,  under the thermodynamic limit, we study the behavior of two point  observables of one dimensional Ising model without external field. By means of \emph{tranfer matrix} method, we get following results.
\begin{theorem}\label{thm:3.1}
	In the case that $h=0$, assume that $ f,g:\{+1,-1\}\rightarrow \mathds{R} $, and $x_{1}=i, x_{2}=j$. Let $\mathcal{S}^{fg}$ denote the  observable induced by $fg$, then by the definition $\mathcal{S}^{fg} := \hat{\mathcal{S}}^{fg} - \tilde{\mathcal{S}}^{fg}$, we can get
	\begin{align*}
		\hat{\mathcal{S}}^{fg}=\frac{A}{4}+\frac{B}{4} \left (\frac{\lambda_{-}}{\lambda_{+}}\right)^{|x_{2}-x_{1}|},
		\tilde{\mathcal{S}}^{fg}=\left(\frac{C^{4}_{12}C^{2}_{12}}{4}+\frac{\bar{\vartriangle}}{4}\left (\frac{\lambda_{-}}{\lambda_{+}}\right)^{|x_{2}-x_{1}|}\right)\log\left(\frac{C^{4}_{12}C^{2}_{12}}{4}+\frac{\bar{\vartriangle}}{4}\left (\frac{\lambda_{-}}{\lambda_{+}}\right)^{|x_{2}-x_{1}|}\right),
	\end{align*}
	where \begin{align*}
		A=&a+b+c+d,\ B=a-b-c+d,\ C^{2}_{12} = g^{2}(1)+g^{2} (-1),\
		C^{4}_{12}=f^{2}(1)+f^{2}(-1),\\
		\bar{\vartriangle}=&(f^{2}(1)-f^{2}(-1))(g^{2}(1)-g^{2}(-1)),
	\end{align*}
	and
	\begin{align*}
		a=&f^{2}(1)g^{2}(1)\log(f^{2}(1)g^{2}(1)), d=f^{2}(-1)g^{2}(-1)\log(f^{2}(-1)g^{2}(-1)),\\
		b=&f^{2}(-1)g^{2}(1)\log(f^{2}(-1)g^{2}(1)),	c=f^{2}(1)g^{2}(-1)\log(f^{2}(1)g^{2}(-1)).
	\end{align*}
\end{theorem}
\begin{proof}
	Let's consider $\hat{\mathcal{S}}^{fg}$ firstly, without loss of generality, we may assume that $x_{1}=i<j=x_{2}$. By definition~\eqref{2.13}~, we know that
	\begin{align*}
		\hat{\mathcal{S}}_{N}^{fg}=& \langle(f^{2}(\sigma_{x_{1}})g^{2}(\sigma_{x_{2}}))\log(f^{2}(\sigma_{x_{1}})g^{2}(\sigma_{x_{2}}))\rangle_{\mathds{P}_{N}}\\
		= & \frac{1}{Z_{N}}\sum_{\{\sigma_{k}=\pm 1\}}(f^{2}(\sigma_{i})g^{2}(\sigma_{j}))\log(f^{2}(\sigma_{i})g^{2}(\sigma_{j}))\exp\{K \sum_{k=1}^{N}\sigma_{k}\sigma_{k+1}\}\\
		= &  \frac{1}{Z_{N}}\sum_{\{\sigma_{k}=\pm 1\}}(f^{2}(\sigma_{i})g^{2}(\sigma_{j}))\log(f^{2}(\sigma_{i}))\exp\{K \sum_{k=1}^{N}\sigma_{k}\sigma_{k+1}\}  \\
		&+\frac{1}{Z_{N}}\sum_{\{\sigma_{k}=\pm 1\}}(f^{2}(\sigma_{i})g^{2}(\sigma_{j}))\log(g^{2}(\sigma_{j}))\exp\{K \sum_{k=1}^{N}\sigma_{k}\sigma_{k+1}\}
		:=  I_{1}+I_{2}.
	\end{align*}
	Define the \emph{tranfer matrix} $P$ as an operator such that
	$$ \langle \sigma_{k} | P | \sigma_{k+1} \rangle = e^{K \sigma_{k}\sigma_{k+1}} \qquad  k\neq i-1,i,j-1,j.$$
	and  the other four operators
	\begin{align*}
		\langle \sigma_{i-1} | C^{1} | \sigma_{i+1} \rangle = & \sum_{\sigma_{i}=\pm 1} f^{2}(\sigma_{i})\log(f^{2}(\sigma_{i}))e^{K \sigma_{i}(\sigma_{i-1}+\sigma_{i+1})},\langle \sigma_{j-1} | C^{2} | \sigma_{j+1} \rangle =  \sum_{\sigma_{j}=\pm 1} g^{2}(\sigma_{j})e^{K \sigma_{j}(\sigma_{j-1}+\sigma_{j+1})},\\
		\langle \sigma_{j-1} | C^{3} | \sigma_{j+1} \rangle = & \sum_{\sigma_{j}=\pm 1} g^{2}(\sigma_{j})\log(g^{2}(\sigma_{j}))e^{K \sigma_{j}(\sigma_{j-1}+\sigma_{j+1})},
		\langle \sigma_{i-1} | C^{4} | \sigma_{i+1} \rangle = \sum_{\sigma_{i}=\pm 1} f^{2}(\sigma_{i})e^{K \sigma_{i}(\sigma_{i-1}+\sigma_{i+1})}.
	\end{align*}
	It will be convenient to express these operators as real symmetric matrices
	\begin{equation*}
		P=\left(                 
		\begin{array}{cc}   
			e^{K} & e^{-K} \\  
			e^{-K} & e^{K}  \\  
		\end{array}
		\right)                 
		,\qquad
		C^{k}=\left(                 
		\begin{array}{cc}   
			C^{k}_{11} & C^{k}_{12} \\  
			C^{k}_{21} & C^{k}_{22}  \\  
		\end{array}
		\right)                 
		\qquad   k=1,2,3,4,
	\end{equation*}
	where
	\begin{align}\label{eq:3.2}
		C^{1}_{11}&=f^{2}(1)\log(f^{2}(1))e^{2K}+f^{2}(-1)\log(f^{2}(-1))e^{-2K},\nonumber\\
		C^{1}_{12}&=C^{1}_{21}=f^{2}(1)\log(f^{2}(1))+f^{2}(-1)\log(f^{2}(-1)),\nonumber\\
		C^{1}_{22}&=f^{2}(1)\log(f^{2}(1))e^{-2K}+f^{2}(-1)\log(f^{2}(-1))e^{2K},\nonumber\\
		C^{3}_{11}&=g^{2}(1)\log(g^{2}(1))e^{2K}+g^{2}(-1)\log(g^{2}(-1))e^{-2K},\nonumber\\
		C^{3}_{12}&=C^{3}_{21}=g^{2}(1)\log(g^{2}(1))+g^{2}(-1)\log(g^{2}(-1)),\nonumber\\
		C^{3}_{22}&=g^{2}(1)\log(g^{2}(1))e^{-2K}+g^{2}(-1)\log(g^{2}(-1))e^{2K},\nonumber\\
		C^{2}_{11}&=g^{2}(1)e^{2K}+g^{2}(-1)e^{-2K},
		C^{2}_{12}=C^{2}_{21}=g^{2}(1)+g^{2}(-1),\\
		C^{2}_{22}&=g^{2}(1)e^{-2K}+g^{2}(-1)e^{2K},
		C^{4}_{11}=f^{2}(1)e^{2K}+f^{2}(-1)e^{-2K},\nonumber\\
		C^{4}_{12}&=C^{4}_{21}=f^{2}(1)+f^{2}(-1),
		C^{4}_{22}=f^{2}(1)e^{-2K}+f^{2}(-1)e^{2K}.\nonumber
	\end{align}
	Hence
	\begin{align*}
		I_{1}=\frac{1}{Z_{N}}\sum_{\{\sigma_{k}=\pm1 ,k\neq i,j \}}  & \langle\sigma_{1}|P|\sigma_{2} \rangle
		\cdots \langle \sigma_{i-2}|P|\sigma_{i-1} \rangle \langle \sigma_{i-1} | C^{1} | \sigma_{i+1}\rangle \langle \sigma_{i+1} | P | \sigma_{i+2} \rangle\\
		&\cdots\langle \sigma_{j-2} | P | \sigma_{j-1} \rangle\langle \sigma_{j-1} | C^{2} | \sigma_{j+1} \rangle \langle \sigma_{j+1} | P | \sigma_{j+2}\rangle\cdots\langle \sigma_{N} | P | \sigma_{1}\rangle.
	\end{align*}
	So $$I_{1}=\frac{1}{Z_{N}}Tr\{P^{i-2}C^{1}P^{j-i-2}C^{2}P^{N-j}\}=\frac{1}{Z_{N}}Tr\{P^{N-j+i-2}C^{1}P^{j-i-2}C^{2}\},$$
	and similarly
	$$I_{2}=\frac{1}{Z_{N}}Tr\{P^{N-j+i-2}C^{4}P^{j-i-2}C^{3}\}.$$
	Because $h=0$, $\phi=-\frac{\pi}{4}$, then 
	\begin{equation}
		\label{2.9}	OPO^{T}=P^{'}\\
		=\left(                 
		\begin{array}{cc}   
			e^{K}+e^{-K} & 0 \\  
			0 & e^{K}-e^{-K}  \\  
		\end{array}
		\right)
		=
		\left(                 
		\begin{array}{cc}   
			\lambda_{+} & 0 \\  
			0 & \lambda_{-}  \\  
		\end{array}
		\right),                 
	\end{equation}
	and
	\begin{equation*}
		O=\left(                 
		\begin{array}{cc}   
			\cos\phi & -\sin\phi \\  
			\sin\phi & \cos\phi  \\  
		\end{array}
		\right)
		=
		\left(                 
		\begin{array}{cc}   
			\frac{\sqrt{2}}{2} & \frac{\sqrt{2}}{2} \\  
			-\frac{\sqrt{2}}{2}& \frac{\sqrt{2}}{2}  \\  
		\end{array}
		\right).                 
	\end{equation*}
	We have		$$I_{1}=\frac{1}{Z_{N}}Tr\{{P^{'}}^{N-j+i-2} OC^{1}O^{T}{P^{'}}^{j-i-2}OC^{2}O^{T}\}, I_{2}=\frac{1}{Z_{N}}Tr\{{P^{'}}^{N-j+i-2} OC^{4}O^{T}{P^{'}}^{j-i-2}OC^{3}O^{T}\}.$$
	With some long but  straightforward algebra we obtain
	\begin{align}\label{3.10}
		OC^{k}O^{T}&=\left(                 
		\begin{array}{cc}   
			C^{k}_{11}\cos^{2}\phi+C^{k}_{22}\sin^{2}\phi-C^{k}_{12}\sin2\phi & \frac{C^{k}_{11}-C^{k}_{22}}{2}\sin2\phi+ C^{k}_{12}\cos2\phi\\  \\
			\frac{C^{k}_{11}-C^{k}_{22}}{2}\sin2\phi+ C^{k}_{12}\cos2\phi & C^{k}_{11}\sin^{2}\phi+C^{k}_{22}\cos^{2}\phi+C^{k}_{12}\sin2\phi  \\  
		\end{array}
		\right)\nonumber\\
		&=\left(                 
		\begin{array}{cc}   
			\frac{C^{k}_{11}+C^{k}_{22}}{2}+C^{k}_{12} & -\frac{C^{k}_{11}-C^{k}_{22}}{2}\\  \\
			-\frac{C^{k}_{11}-C^{k}_{22}}{2} & \frac{C^{k}_{11}+C^{k}_{22}}{2}-C^{k}_{12} \\  
		\end{array}
		\right):=\left(                 
		\begin{array}{cc}   
			\hat{C}^{k}_{11}
			& \hat{C}^{k}_{12}\\\hat{C}^{k}_{21} & \hat{C}^{k}_{22} \\  
		\end{array}	\right),\qquad   k=1,2,3,4.
	\end{align}
	Then
	\begin{align*}
		I_{1}=\frac{1}{Z_{N}}Tr\{&
		\left(                 
		\begin{array}{cc}   
			\hat{C}^{1}_{11}\lambda^{N-4}_{+}   & \hat{C}^{1}_{12}\lambda^{N-j+i-2}_{+}\lambda^{j-i-2} _{-}\\  \\
			\hat{C}^{1}_{12}\lambda^{N-j+i-2}_{-}\lambda^{j-i-2} _{+} & \hat{C}^{1}_{22}\lambda^{N-4} _{-} \\  
		\end{array}
		\right)
		\left(                 
		\begin{array}{cc}   
			\hat{C}^{2}_{11}& \hat{C}^{2}_{12}\\
			\hat{C}^{2}_{21} & \hat{C}^{2}_{22}
		\end{array}
		\right)\}.
	\end{align*}
	So we obtain
	\begin{align*}
		I_{1}=\frac{1}{\lambda^{N}_{+}+\lambda^{N}_{-}}
		\{&\hat{C}^{1}_{11}\hat{C}^{2}_{11}\lambda^{N-4}_{+} +\hat{C}^{1}_{12}\hat{C}^{2}_{12}\lambda^{N-j+i-2}_{+}\lambda^{j-i-2} _{-}+\hat{C}^{1}_{12}\hat{C}^{2}_{12}\lambda^{N-j+i-2}_{-}\lambda^{j-i-2} _{+}+\hat{C}^{1}_{22}\hat{C}^{2}_{22}\lambda^{N-4} _{-}\}.
	\end{align*}
	Recall that $0\leq\lambda_{-}<\lambda_{+}$, then
	\begin{align*}
		\lim_{N \to \infty}I_{1}
		=\hat{C}^{1}_{11}\hat{C}^{2}_{11}\lambda^{-4}_{+}   +\hat{C}^{1}_{12}\hat{C}^{2}_{12}\left(\frac{\lambda _{-}}{\lambda_{+}}\right)^{j-i} \lambda^{-2}_{+}\lambda^{-2}_{-}.
	\end{align*}
	Using the symbols introduced in equation~\eqref{eq:3.2}~, it's easy to see
	$$\lim_{N \to \infty}I_{1}=\frac{1}{4}C^{1}_{12}C^{2}_{12}
	+\frac{\hat{\vartriangle}}{4}\left(\frac{\lambda _{-}}{\lambda_{+}}\right)^{j-i}, \lim_{N \to \infty}I_{2}=\frac{1}{4}C^{3}_{12}C^{4}_{12}
	+\frac{\tilde{\vartriangle}}{4}\left(\frac{\lambda _{-}}{\lambda_{+}}\right)^{j-i},$$
	where
	\begin{align*}
		&\hat{\vartriangle}=\left(f^{2}(1)\log(f^{2}(1))-f^{2}(-1)\log(f^{2}(-1)) \right)\left (g^{2}(1)-g^{2}(-1)\right),\\
		&\tilde{\vartriangle}=\left(g^{2}(1)\log(g^{2}(1))-g^{2}(-1)\log(g^{2}(-1)) \right)\left (f^{2}(1)-f^{2}(-1)\right).
	\end{align*}
	It follows that
	$$\hat{\mathcal{S}}^{fg}=\lim_{N \to \infty}I_{1}+I_{2}=\frac{A}{4}+\frac{B}{4} \left (\frac{\lambda_{-}}{\lambda_{+}}\right)^{j-i}.$$
	Now, let's begin to calculate the second term $\tilde{\mathcal{S}}^{fg}$,
	\begin{align*}
		\|fg\|^{2}_{L^{2}(\{\sigma\},\mathds{P}_{N})}=& \frac{1}{Z_{N}}\sum_{\{\sigma_{k}=\pm 1\}}f^{2}(\sigma_{i})g^{2}(\sigma_{j})\exp\{K \sum_{k=1}^{N}\sigma_{k}\sigma_{k+1}\}\\
		=&\frac{1}{Z_{N}}Tr\{{P^{'}}^{N-j+i-2} OC^{4}O^{T}{P^{'}}^{j-i-2}OC^{2}O^{T}\},
	\end{align*}
	the remainder of the calculation can be finished in the same way as shown before, and 
	$$\lim_{N \to \infty}\|fg\|^{2}_{L^{2}(\{\sigma\},\mathds{P}_{N})}=\frac{1}{4}C^{4}_{12}C^{2}_{12}
	+\frac{\bar{\vartriangle}}{4}\left(\frac{\lambda _{-}}{\lambda_{+}}\right)^{j-i},$$
	where
	$$\bar{\vartriangle}=\left (g^{2}(1)-g^{2}(-1)\right)\left (f^{2}(1)-f^{2}(-1)\right).$$
	Finally,
	\begin{align*}
		\tilde{\mathcal{S}}^{fg}=&\lim_{N \to \infty}\|fg\|^{2}_{L^{2}(\{\sigma\},\mathds{P}_{N})}\log(\|fg\|^{2}_{L^{2}(\{\sigma\},\mathds{P}_{N})})\\
		&=\left(\frac{C^{4}_{12}C^{2}_{12}}{4}+\frac{\bar{\vartriangle}}{4}\left (\frac{\lambda_{-}}{\lambda_{+}}\right)^{j-i}\right)\log\left(\frac{C^{4}_{12}C^{2}_{12}}{4}+\frac{\bar{\vartriangle}}{4}\left (\frac{\lambda_{-}}{\lambda_{+}}\right)^{j-i}\right).
	\end{align*}
	which completes the proof of this theorem.
\end{proof}
\paragraph{Remark.}\label{re3.2}
Notice that the behavior of two point correlation function is similar to the two point  observable  as $|x_{2}-x_{1}|\rightarrow \infty$, both of them decay exponentially.

Under the periodic boundary conditions, we get results above for one dimensional Ising model without external field. The methods we used can be extended to treat other boundary conditions. For any finite subset $\Lambda\subset \mathds{Z}$, denote that $\partial_{\Lambda}$ is the boundary of $\Lambda$. In our paper, we consider a special case $\Lambda=\{1,2,\cdots,N\}$ and $\partial_{\Lambda}=\{0,N+1\}$. When $h=0$, for one dimensional Ising model with nearest neighbor interaction, the Gibbs measure associated to it is given by	
\begin{align}\label{eq:3.5}	\mu_{\Lambda}=\frac{1}{Z_{\Lambda}}\exp\{K\sigma_{0}\sigma_{1}+K\sum\limits_{x\in \Lambda }\sigma_{x}\sigma_{x+1}\},
\end{align}
where $Z_{\Lambda}$ is the partition function
\begin{align}\label{eq:3.6}	Z_{\Lambda}=\sum_{\sigma_{x},x\in\Lambda}\exp\{K\sigma_{0}\sigma_{1}+K\sum\limits_{x\in \Lambda }\sigma_{x}\sigma_{x+1}\}.
\end{align}
Define four sets of spins as follows
\begin{align*}
	\partial_{\Lambda_{1}}&=(\sigma_{0}=1,\sigma_{N+1}=1),	\partial_{\Lambda_{2}}=(\sigma_{0}=-1,\sigma_{N+1}=1),\\
	\partial_{\Lambda_{3}}&=(\sigma_{0}=1,\sigma_{N+1}=-1),\partial_{\Lambda_{4}}=(\sigma_{0}=-1,\sigma_{N+1}=-1),
\end{align*}
and $Z_{\partial_{\Lambda_{i}}}$ represents the partition function of system under  boundary conditions $\partial_{\Lambda_{i}},i=1,2,3,4$. 
\begin{theorem}\label{thm:3.2}
	Under the same assumption in theorem \ref{thm:3.1}. Let $\mathcal{S}^{fg,\partial_{\Lambda}}$ represents the  observable  induced by $fg$ with respect to $\mu_{\Lambda}$ under boundary condition $\partial_{\Lambda}$, which can be defined by $\mathcal{S}^{fg,\partial_{\Lambda}} := \hat{\mathcal{S}}^{fg,\partial_{\Lambda}} - \tilde{\mathcal{S}}^{fg,\partial_{\Lambda}}$, then we can get
	\begin{align*}
		\left(
		\begin{array}{cc}   
			\hat{\mathcal{S}}^{fg,\partial_{\Lambda_{1}}} & \hat{\mathcal{S}}^{fg,\partial_{\Lambda_{3}}}\\  
			\hat{\mathcal{S}}^{fg,\partial_{\Lambda_{2}}}& \hat{\mathcal{S}}^{fg,\partial_{\Lambda_{4}}} \\  
		\end{array}
		\right)
		&=\left(                 
		\begin{array}{cc}   
			(L_{11}-L_{21}+M_{11}-M_{21})\lambda_{+}^{-j-1}	& (L_{11}-L_{21}+M_{11}-M_{21})\lambda_{+}^{-j-1}\\  
			(L_{11}+L_{21}+M_{11}+M_{21})\lambda_{+}^{-j-1}& (L_{11}+L_{21}+M_{11}+M_{21})\lambda_{+}^{-j-1} \\  
		\end{array}
		\right),
	\end{align*}
	and
	\begin{align*}
		\left(
		\begin{array}{cc}   
			\tilde{\mathcal{S}}^{fg,\partial_{\Lambda_{1}}} & \tilde{\mathcal{S}}^{fg,\partial_{\Lambda_{3}}}\\  
			\tilde{\mathcal{S}}^{fg,\partial_{\Lambda_{2}}}& \tilde{\mathcal{S}}^{fg,\partial_{\Lambda_{4}}} \\  
		\end{array}
		\right)
		&=\left(                 
		\begin{array}{cc}   
			\hat{R}_{1}& \hat{R}_{1}\\  
			\hat{R}_{2}& \hat{R}_{2} \\  
		\end{array}
		\right),
	\end{align*}
	where $M_{11},L_{11},\hat{R}_{1},M_{21},L_{21},\hat{R}_{12}$ are constants depend on $f,g$ and $\lambda_{+},\lambda_{-}$.	
\end{theorem}
\begin{proof}
	By means of the \emph{tranfer matrix} method we used in theorem \ref{thm:3.1}, equation~\eqref{eq:3.6}~can be written as
	\begin{align*}			Z_{\Lambda}=\sum\limits_{\{\sigma_{k}=\pm1,k\in\Lambda\}}  & \langle\sigma_{0}|P|\sigma_{1} \rangle \langle\sigma_{1}|P|\sigma_{2} \rangle
		\cdots \langle \sigma_{N-1} | P | \sigma_{N}\rangle\langle \sigma_{N} | P | \sigma_{N+1}\rangle=\langle\sigma_{0}|P^{N+1}|\sigma_{N+1} \rangle.
	\end{align*}
	Thanks to equation~\eqref{2.9}~
	\begin{align*}			P^{N+1}&=O^{T}
		\left(                 
		\begin{array}{cc}   
			\lambda^{N+1}_{+} & 0 \\  
			0 & \lambda^{N+1}_{-}  \\  
		\end{array}
		\right)O=\frac{1}{2}\left(                 
		\begin{array}{cc}   
			\lambda^{N+1}_{+}+\lambda^{N+1}_{-} & \lambda^{N+1}_{+}-\lambda^{N+1}_{-} \\  
			\lambda^{N+1}_{+}-\lambda^{N+1}_{-} & \lambda^{N+1}_{+}+\lambda^{N+1}_{-}  \\  
		\end{array}
		\right):=\left(                 
		\begin{array}{cc}   
			Z_{\partial_{\Lambda_{1}}} & Z_{\partial_{\Lambda_{3}}} \\  
			Z_{\partial_{\Lambda_{2}}} & Z_{\partial_{\Lambda_{4}}}  
		\end{array}
		\right).
	\end{align*}
	Obviously, for $k=1,2,3,4$, under the thermodynamic limit
	\begin{align*}			F=\lim_{N\rightarrow+\infty}\frac{\log(Z_{\partial_{\Lambda_{k}}})}{N+1}=\lim_{N\rightarrow+\infty}\frac{\log\lambda^{N+1}_{+}+\log(1\pm(\frac{\lambda_{-}}{\lambda_{+}})^{N+1})}{N+1}=\log\lambda_{+},
	\end{align*}
	which is independent to the boundary conditions $\partial_{\Lambda_{i}}$.
	
	For the first term
	$\hat{\mathcal{S}}^{fg}$, we also assume that $x_{1}=i<j=x_{2}$. By definition~\eqref{2.13}~, for $k=1,2,3,4$, the key point is to calculate terms
	\begin{align*}
		\hat{\mathcal{S}}_{N}^{fg,\partial_{\Lambda_{k}}}=&\frac{1}{Z_{\partial_{\Lambda_{k}}}}\sum_{\{\sigma_{x}=\pm 1,x\in \Lambda\}}(f^{2}(\sigma_{i})g^{2}(\sigma_{j}))\log(f^{2}(\sigma_{i})g^{2}(\sigma_{j}))\exp\{K\sigma_{0}\sigma_{1}+K\sum\limits_{x\in \Lambda }\sigma_{x}\sigma_{x+1}\}\\
		= &  \frac{1}{Z_{\partial_{\Lambda_{k}}}}\sum_{\{\sigma_{x}=\pm 1,x\in \Lambda\}}(f^{2}(\sigma_{i})g^{2}(\sigma_{j}))\log(f^{2}(\sigma_{i}))\exp\{K\sigma_{0}\sigma_{1}+K\sum\limits_{x\in \Lambda }\sigma_{x}\sigma_{x+1}\}  \\
		+&\frac{1}{Z_{\partial_{\Lambda_{k}}}}\sum_{\{\sigma_{x}=\pm 1,x\in \Lambda\}}(f^{2}(\sigma_{i})g^{2}(\sigma_{j}))\log(g^{2}(\sigma_{j}))\exp\{K\sigma_{0}\sigma_{1}+K\sum\limits_{x\in \Lambda }\sigma_{x}\sigma_{x+1}\}
		:= I^{k}_{1}+I^{k}_{2}.
	\end{align*}
	Using the symbols we introduced in theorem \ref{thm:3.1}, under the boundary condition $\partial_{\Lambda_{k}}$
	\begin{align*}
		I^{k}_{1}=\frac{1}{Z_{\partial_{\Lambda_{k}}}}\sum_{\begin{subarray}{l}
				\sigma_{x}=\pm1 ,\\
				x\neq i,j,\in \Lambda
		\end{subarray}} 
		& \langle\sigma_{0}|P|\sigma_{1} \rangle\langle\sigma_{1}|P|\sigma_{2} \rangle
		\cdots \langle \sigma_{i-2}|P|\sigma_{i-1} \rangle \langle \sigma_{i-1} | C^{1} | \sigma_{i+1}\rangle \langle \sigma_{i+1} | P | \sigma_{i+2} \rangle\\
		&\cdots\langle \sigma_{j-2} | P | \sigma_{j-1} \rangle\langle \sigma_{j-1} | C^{2} | \sigma_{j+1} \rangle \langle \sigma_{j+1} | P | \sigma_{j+2}\rangle\cdots\langle \sigma_{N} | P | \sigma_{N+1}\rangle,\\
		I^{k}_{2}=\frac{1}{Z_{\partial_{\Lambda_{k}}}}\sum_{\begin{subarray}{l}
				\sigma_{x}=\pm1 ,\\
				x\neq i,j,\in \Lambda
		\end{subarray}} 
		& \langle\sigma_{0}|P|\sigma_{1} \rangle\langle\sigma_{1}|P|\sigma_{2} \rangle
		\cdots \langle \sigma_{i-2}|P|\sigma_{i-1} \rangle \langle \sigma_{i-1} | C^{4} | \sigma_{i+1}\rangle \langle \sigma_{i+1} | P | \sigma_{i+2} \rangle\\
		&\cdots\langle \sigma_{j-2} | P | \sigma_{j-1} \rangle\langle \sigma_{j-1} | C^{3} | \sigma_{j+1} \rangle \langle \sigma_{j+1} | P | \sigma_{j+2}\rangle\cdots\langle \sigma_{N} | P | \sigma_{N+1}\rangle.
	\end{align*}
	Hence, for $\partial_{\Lambda_{k}}=(\sigma_{0},\sigma_{N+1})$
	\begin{align*}
		I^{k}_{1} &=\frac{1}{Z_{\partial_{\Lambda_{k}}}}
		\langle\sigma_{0}|P^{i-1} C^{1} P^{j-i-2}C^{2}P^{N-j} | \sigma_{N+1}\rangle,
		I^{k}_{2}=\frac{1}{Z_{\partial_{\Lambda_{k}}}}
		\langle\sigma_{0}|P^{i-1} C^{4} P^{j-i-2}C^{3}P^{N-j} | \sigma_{N+1}\rangle.
	\end{align*}
	Noticed that
	\begin{align*}
		P^{i-1} C^{1} P^{j-i-2}C^{2}P^{N-j}&=O^{T}
		{P^{'}}^{i-1}OC^{1}O^{T}
		{P^{'}}^{j-i-2}OC^{2}O^{T}{P^{'}}^{N-j}O,\\
		P^{i-1} C^{4} P^{j-i-2}C^{3}P^{N-j}&=O^{T}
		{P^{'}}^{i-1}OC^{4}O^{T}
		{P^{'}}^{j-i-2}OC^{3}O^{T}{P^{'}}^{N-j}O   .
	\end{align*}
	Let
	\begin{align*}
		M&=\left(                 
		\begin{array}{cc}   
			M_{11} & M_{12} \\  
			M_{21}& M_{22} \\  
		\end{array}
		\right)={P^{'}}^{i-1}OC^{1}O^{T}
		{P^{'}}^{j-i-2}OC^{2}O^{T},\\
		L&=\left(                 
		\begin{array}{cc}   
			L_{11} & L_{12} \\  
			L_{21}& L_{22} \\  
		\end{array}
		\right)={P^{'}}^{i-1}OC^{4}O^{T}
		{P^{'}}^{j-i-2}OC^{3}O^{T},\\
		R&=\left(                 
		\begin{array}{cc}   
			R_{11} & R_{12} \\  
			R_{21}& R_{22} \\  
		\end{array}
		\right)={P^{'}}^{i-1}OC^{4}O^{T}
		{P^{'}}^{j-i-2}OC^{2}O^{T}.
	\end{align*}
	Then 
	\begin{align*}
		P^{i-1}& C^{1} P^{j-i-2}C^{2}P^{N-j}\\
		&=O^{T}M{P^{'}}^{N-j}O=\left(
		\begin{array}{cc}   
			\frac{\sqrt{2}}{2} & -\frac{\sqrt{2}}{2} \\  
			\frac{\sqrt{2}}{2}& \frac{\sqrt{2}}{2}  \\  
		\end{array}
		\right)\left(                 
		\begin{array}{cc}   
			M_{11} & M_{12} \\  
			M_{21}& M_{22} \\  
		\end{array}
		\right)\left(                 
		\begin{array}{cc}   
			\lambda_{+}^{N-j} & 0 \\  
			0& \lambda_{-}^{N-j} \\  
		\end{array}
		\right)\left(
		\begin{array}{cc}   
			\frac{\sqrt{2}}{2} & \frac{\sqrt{2}}{2} \\  
			-\frac{\sqrt{2}}{2}& \frac{\sqrt{2}}{2}  \\  
		\end{array}
		\right)\\
		&=\frac{1}{2}\left(                 
		\begin{array}{cc}   
			(M_{11}-M_{21})\lambda_{+}^{N-j}-(M_{12}-M_{22})\lambda_{-}^{N-j} & (M_{11}-M_{21})\lambda_{+}^{N-j}+(M_{12}-M_{22})\lambda_{-}^{N-j} \\  
			(M_{21}+M_{11})\lambda_{+}^{N-j}+(M_{22}+M_{12})\lambda_{-}^{N-j}& (M_{21}+M_{11})\lambda_{+}^{N-j}+(M_{22}+M_{12})\lambda_{-}^{N-j} \\  
		\end{array}
		\right).
	\end{align*}
	Due to~\eqref{3.10}~, 
	\begin{align*}
		\left(                 
		\begin{array}{cc}   
			M_{11} & M_{12} \\  
			M_{21}& M_{22} \\  
		\end{array}
		\right)&=\left(                 
		\begin{array}{cc}   
			\lambda_{+}^{i-1} & 0 \\  
			0& \lambda_{-}^{i-1} \\  
		\end{array}
		\right)\left(                 
		\begin{array}{cc}   
			\hat{C}^{1}_{11}
			& \hat{C}^{1}_{12}\\\hat{C}^{1}_{21} & \hat{C}^{1}_{22}  
		\end{array}	\right)
		\left(                 
		\begin{array}{cc}   
			\lambda_{+}^{j-i-2} & 0 \\  
			0& \lambda_{-}^{j-i-2}   
		\end{array}
		\right)\left(                 
		\begin{array}{cc}   
			\hat{C}^{2}_{11}
			& \hat{C}^{2}_{12}\\\hat{C}^{2}_{21} & \hat{C}^{2}_{22}   
		\end{array}	\right)\\
		&=\left(                 
		\begin{array}{cc}   
			\hat{C}^{1}_{11}\lambda_{+}^{i-1}
			& \hat{C}^{1}_{12}\lambda_{+}^{i-1}\\ \hat{C}^{1}_{21}\lambda_{-}^{i-1} & \hat{C}^{1}_{22}\lambda_{-}^{i-1}
		\end{array}	\right)
		\left(                 
		\begin{array}{cc}   
			\hat{C}^{2}_{11}\lambda_{+}^{j-i-2} 
			& \hat{C}^{2}_{12}\lambda_{+}^{j-i-2} \\\hat{C}^{2}_{21}\lambda_{-}^{j-i-2} & \hat{C}^{2}_{22}\lambda_{-}^{j-i-2}   
		\end{array}	\right)\\
		&=\left(                 
		\begin{array}{cc}   
			\hat{C}^{1}_{11}\hat{C}^{2}_{11}\lambda_{+}^{j-3}+\hat{C}^{1}_{12}\hat{C}^{2}_{21}\lambda_{+}^{i-1}\lambda_{-}^{j-i-2}
			& \hat{C}^{1}_{11}\hat{C}^{2}_{12}\lambda_{+}^{i-1}\lambda_{-}^{j-i-2}+\hat{C}^{1}_{12}\hat{C}^{2}_{22}\lambda_{+}^{i-1}\lambda_{-}^{j-i-2}\\ \hat{C}^{1}_{21}\hat{C}^{2}_{11}\lambda_{+}^{j-i-2}\lambda_{-}^{i-1}+\hat{C}^{1}_{22}\hat{C}^{2}_{21}\lambda_{-}^{j-3} & \hat{C}^{1}_{21}\hat{C}^{2}_{12}\lambda_{+}^{j-i-2}\lambda_{-}^{i-1}+\hat{C}^{1}_{22}\hat{C}^{2}_{22}\lambda_{-}^{j-3}
		\end{array}	\right),
	\end{align*}
	and
	\begin{align*}
		\left(                 
		\begin{array}{cc}   
			L_{11} & L_{12} \\  
			L_{21}& L_{22} \\  
		\end{array}
		\right)
		&=\left(                 
		\begin{array}{cc}   
			\hat{C}^{4}_{11}\hat{C}^{3}_{11}\lambda_{+}^{j-3}+\hat{C}^{4}_{12}\hat{C}^{3}_{21}\lambda_{+}^{i-1}\lambda_{-}^{j-i-2}
			& \hat{C}^{4}_{11}\hat{C}^{3}_{12}\lambda_{+}^{i-1}\lambda_{-}^{j-i-2}+\hat{C}^{4}_{12}\hat{C}^{3}_{22}\lambda_{+}^{i-1}\lambda_{-}^{j-i-2}\\ \hat{C}^{4}_{21}\hat{C}^{3}_{11}\lambda_{+}^{j-i-2}\lambda_{-}^{i-1}+\hat{C}^{4}_{22}\hat{C}^{3}_{21}\lambda_{-}^{j-3} & \hat{C}^{4}_{21}\hat{C}^{3}_{12}\lambda_{+}^{j-i-2}\lambda_{-}^{i-1}+\hat{C}^{4}_{22}\hat{C}^{3}_{22}\lambda_{-}^{j-3}
		\end{array}	\right),\\
		\left(                 
		\begin{array}{cc}   
			R_{11} & R_{12} \\  
			R_{21}& R_{22} \\  
		\end{array}
		\right)
		&=\left(                 
		\begin{array}{cc}   
			\hat{C}^{4}_{11}\hat{C}^{2}_{11}\lambda_{+}^{j-3}+\hat{C}^{4}_{12}\hat{C}^{2}_{21}\lambda_{+}^{i-1}\lambda_{-}^{j-i-2}
			& \hat{C}^{4}_{11}\hat{C}^{2}_{12}\lambda_{+}^{i-1}\lambda_{-}^{j-i-2}+\hat{C}^{4}_{12}\hat{C}^{2}_{22}\lambda_{+}^{i-1}\lambda_{-}^{j-i-2}\\ \hat{C}^{4}_{21}\hat{C}^{2}_{11}\lambda_{+}^{j-i-2}\lambda_{-}^{i-1}+\hat{C}^{4}_{22}\hat{C}^{2}_{21}\lambda_{-}^{j-3} & \hat{C}^{4}_{21}\hat{C}^{2}_{12}\lambda_{+}^{j-i-2}\lambda_{-}^{i-1}+\hat{C}^{4}_{22}\hat{C}^{2}_{22}\lambda_{-}^{j-3}
		\end{array}	\right).
	\end{align*}
	Finally
	\begin{align}\label{eq:3.13}
		\left(         
		\begin{array}{cc}   
			I_{1}^{1} & I_{1}^{3}\\  
			I_{1}^{2}& I_{1}^{4} \\  
		\end{array}
		\right) 
		&=\frac{1}{2}\left(                 
		\begin{array}{cc}   
			\frac{(M_{11}-M_{21})\lambda_{+}^{N-j}-(M_{12}-M_{22})\lambda_{-}^{N-j} }{Z_{\partial_{\Lambda_{1}}}}	& \frac{(M_{11}-M_{21})\lambda_{+}^{N-j}+(M_{12}-M_{22})\lambda_{-}^{N-j}} {Z_{\partial_{\Lambda_{3}}}}\\  
			\frac{(M_{21}+M_{11})\lambda_{+}^{N-j}+(M_{22}+M_{12})\lambda_{-}^{N-j}}{Z_{\partial_{\Lambda_{2}}}}& \frac{(M_{21}+M_{11})\lambda_{+}^{N-j}+(M_{22}+M_{12})\lambda_{-}^{N-j}}{Z_{\partial_{\Lambda_{4}}}} \\  
		\end{array}
		\right).
	\end{align}
	It's easy to see
	\begin{align*}
		\lim_{N \to \infty}\left(         
		\begin{array}{cc}   
			I_{1}^{1} & I_{1}^{3}\\  
			I_{1}^{2}& I_{1}^{4} \\  
		\end{array}
		\right) 
		&=\left(                 
		\begin{array}{cc}   
			(M_{11}-M_{21})\lambda_{+}^{-j-1}	& (M_{11}-M_{21})\lambda_{+}^{-j-1}\\  
			(M_{11}+M_{21})\lambda_{+}^{-j-1}& (M_{11}+M_{21})\lambda_{+}^{-j-1} \\  
		\end{array}
		\right).
	\end{align*}
	Similarly,
	\begin{align*}
		\lim_{N \to \infty}\left(         
		\begin{array}{cc}   
			I_{2}^{1} & I_{2}^{3}\\  
			I_{2}^{2}& I_{2}^{4} \\  
		\end{array}
		\right) 
		&=\left(                 
		\begin{array}{cc}   
			(L_{11}-L_{21})\lambda_{+}^{-j-1}	& (L_{11}-L_{21})\lambda_{+}^{-j-1}\\  
			(L_{11}+L_{21})\lambda_{+}^{-j-1}& (L_{11}+L_{21})\lambda_{+}^{-j-1} \\  
		\end{array}
		\right).
	\end{align*}
	Then under the thermodynamic limit
	\begin{align*}
		\left(
		\begin{array}{cc}   
			\hat{\mathcal{S}}^{fg,\partial_{\Lambda_{1}}} & \hat{\mathcal{S}}^{fg,\partial_{\Lambda_{3}}}\\  
			\hat{\mathcal{S}}^{fg,\partial_{\Lambda_{2}}}& \hat{\mathcal{S}}^{fg,\partial_{\Lambda_{4}}} \\  
		\end{array}
		\right)
		&=\left(                 
		\begin{array}{cc}   
			(L_{11}-L_{21}+M_{11}-M_{21})\lambda_{+}^{-j-1}	& (L_{11}-L_{21}+M_{11}-M_{21})\lambda_{+}^{-j-1}\\  
			(L_{11}+L_{21}+M_{11}+M_{21})\lambda_{+}^{-j-1}& (L_{11}+L_{21}+M_{11}+M_{21})\lambda_{+}^{-j-1} \\  
		\end{array}
		\right).
	\end{align*}
	For the second term $\tilde{\mathcal{S}}_{N}^{fg,\partial_{\Lambda_{k}}}$, denote
	\begin{align*}
		J_{k}:=	\|fg\|^{2}_{L^{2}(\{\sigma\},\mu_{\partial_{\Lambda_{k}}})}=&\frac{1}{Z_{\partial_{\Lambda_{k}}}}\sum_{\{\sigma_{x}=\pm 1,x\in \Lambda\}}f^{2}(\sigma_{i})g^{2}(\sigma_{j})\exp\{K\sigma_{0}\sigma_{1}+K\sum\limits_{x\in \Lambda }\sigma_{x}\sigma_{x+1}\}\\
		=&\frac{1}{Z_{\partial_{\Lambda_{k}}}}
		\langle\sigma_{0}|P^{i-1} C^{4} P^{j-i-2}C^{3}P^{N-j} | \sigma_{N+1}\rangle,\ k=1,2,3,4.
	\end{align*}
	which means $\tilde{\mathcal{S}}_{N}^{fg,\partial_{\Lambda_{k}}}=J_{k}\log J_{k}$. By the same way, we  obtain
	\begin{align}\label{3.14}
		\lim_{N \to \infty}\left(         
		\begin{array}{cc}   
			J_{1} & J_{3}\\  
			J_{2}& J_{4} \\  
		\end{array}
		\right) 
		&=\left(                 
		\begin{array}{cc}   
			(R_{11}-R_{21})\lambda_{+}^{-j-1}	& (R_{11}-R_{21})\lambda_{+}^{-j-1}\\  
			(R_{11}+R_{21})\lambda_{+}^{-j-1}& (R_{11}+R_{21})\lambda_{+}^{-j-1} \\  
		\end{array}
		\right).
	\end{align}
	Let
	\begin{align*}
		\hat{R}_{1}=(R_{11}-R_{21})\lambda_{+}^{-j-1}\log\left((R_{11}-R_{21})\lambda_{+}^{-j-1}\right), 
		\hat{R}_{2}=		(R_{11}+R_{21})\lambda_{+}^{-j-1}\log\left((R_{11}+R_{21})\lambda_{+}^{-j-1}\right),
	\end{align*}
	then
	\begin{align*}
		\left(
		\begin{array}{cc}   
			\tilde{\mathcal{S}}^{fg,\partial_{\Lambda_{1}}} & \tilde{\mathcal{S}}^{fg,\partial_{\Lambda_{3}}}\\  
			\tilde{\mathcal{S}}^{fg,\partial_{\Lambda_{2}}}& \tilde{\mathcal{S}}^{fg,\partial_{\Lambda_{4}}} \\  
		\end{array}
		\right)
		&=\left(                 
		\begin{array}{cc}   
			\hat{R}_{1}	& \hat{R}_{1}\\  
			\hat{R}_{2}&\hat{R}_{2} 
		\end{array}
		\right).
	\end{align*}
	
\end{proof}

\section{Rewrite the RG procedure in other principles}\label{sec:4}
From the physical point of view, there are different frames to construct the RG procedure according to the underlying observables we choose. Among these quantities, the most important and general one is the connected $n$-point correlation function; see \cite{Zinn2014book}. Concretely, the main idea is  to construct a new Hamiltonian  $\mathcal{H}_{\lambda}$ such that correlation function satisfies the following scaling formula when the real positive dilatation parameter ${\lambda}$ of spatial points diverges
\begin{itemize}
	\item  $\mathcal{H}_{\lambda = 1}=\mathcal{H}$,
	\item  For any $x_{1},\cdots, x_{n}\in \mathds{R}$, the correlation functions $\mathcal{G}(x_{1},\cdots, x_{n})$ satisfy the following scaling formula
	\begin{align}\label{2.12}
		\mathcal{G}_{\lambda}(x_{1},\cdots, x_{n})-{Z^{-n/2}(\lambda)}\mathcal{G}(\lambda x_{1},\cdots, \lambda x_{n})=\mathcal{R}_{\lambda}(x_{1},\cdots, x_{n}),
	\end{align}
	with
	$$\mathcal{G}_{\lambda=1}(x_{1},\cdots, x_{n})\equiv \mathcal{G}( x_{1},\cdots, x_{n}), \mathcal{R}_{\lambda=1}(x_{1},\cdots, x_{n})=0, Z(1)=1,$$
	and the RG equation~\eqref{2.12}~involves term $\mathcal{R}_{\lambda}$ need to satisfy 
	$$\mathcal{R}_{\lambda}(x_{1},\cdots, x_{n})\rightarrow 0\ as \ \lambda\rightarrow\infty, $$ faster than any power of $\lambda$.
\end{itemize}
The Hamiltonian  $\mathcal{H}_{\lambda}$  is also called the \emph{effective Hamiltonian} at scale $\lambda$. This section we discuss the scaling property of one dimensional Ising model along the RG procedure. Within this framework, the \emph{RG} equation~\eqref{2.12}~mentioned above is essential. In section \ref{sec:4.1},  we rewrite the "coarse-graining" procedure based on $n$-point correlation functions and prove that the \emph{RGT}~\eqref{2.11}~satisfies scaling formula~\eqref{2.12}~under $\lambda=2$ in each iteration. The main results of section \ref{sec:4.2} is to give a compatible principle of \emph{RG} procedure in our new observables words. Specifically, the \emph{RGT} ~\eqref{2.11}~that obtained by the "coarse-graining" procedure coincides with the RG equation~\eqref{2.12}~and observables RG equation we proposed in section \ref{sec:4.2} below.
\subsection{The "coarse-graining" procedure is consistent with the \emph{RG} equations. }\label{sec:4.1}
In the case of two point correlation functions, the \emph{RG} equation~\eqref{2.12}~can be written as
\begin{align}\label{2.16}
	\mathcal{G}_{\lambda}(x_{1}, x_{2})-Z^{-1}(\lambda)\mathcal{G}(\lambda x_{1}, \lambda x_{2})=\mathcal{R}_{\lambda}(x_{1}, x_{2}).
\end{align}
The main goal in this section is to check if the \emph{RGT}~\eqref{2.11}~satisfies the \emph{RG} equation~\eqref{2.16}~through a few iterations. Notice that, in the "coarse-graining" procedure, dilatation parameter $\lambda=2$. We choose the function $Z(\lambda)=1$, and then for any $x_{1},x_{2}$ 
\begin{align*}
	\mathcal{G}_{\lambda=2}(x_{1}, x_{2})-\mathcal{G}(2 x_{1}, 2 x_{2})=\mathcal{R}_{0}(x_{1}, x_{2}).
\end{align*}
For a given $n\in\mathds{N}$, iterating equation~\eqref{2.12}~, and due to the arbitrariness of $x_{1},x_{2}$, we obtain
\begin{align}
	&\mathcal{G}_{\lambda=2}(2^{n-1}x_{1}, 2^{n-1}x_{2})-\mathcal{G}(2^{n} x_{1}, 2^{n} x_{2}):=\mathcal{R}_{1}(x_{1}, x_{2}),\nonumber\\
	& \qquad\qquad  \vdots\\
	&\mathcal{G}_{\lambda=2^{n}}(x_{1}, x_{2})-\mathcal{G}_{\lambda=2^{n-1}}(2 x_{1}, 2 x_{2}):=\mathcal{R}_{n}(x_{1}, x_{2}).\nonumber
\end{align}
Hence
\begin{align}\label{4.18}
	&\mathcal{G}_{\lambda=2^{n}}(x_{1}, x_{2})-\mathcal{G}(2^{n} x_{1}, 2^{n} x_{2})=\sum_{k=1}^{n}\mathcal{R}_{k}(x_{1}, x_{2}):=\mathcal{S}_{n}(x_{1}, x_{2}).
\end{align}
\begin{theorem}\label{thm:4.1}
	For the Ising chain without external  field ($h=0$) on  one dimentional lattice $\mathds{Z}$, Hamiltonians $\mathcal{H}_{\lambda} $ is determined by the \emph{RGT}~\eqref{2.11}~, $\mathcal{G}_{\lambda}(x_{1}, x_{2})$ are the two point correlation functions corresponding to $\mathcal{H}_{\lambda}$. Then the remainder terms $\mathcal{S}_{n}(x_{1}, x_{2})$ decay to 0 exponentially.
\end{theorem}
\begin{proof}
	From results listed in section \ref{sec:2}, due to $h=0$, for any $x_{1},x_{2}\in\mathds{Z}$
	\begin{align*}
		\mathcal{G}(x_{1},x_{2})&=\left (\frac{\lambda_{-}}{\lambda_{+}}\right)^{|x_{2}-x_{1}|},
	\end{align*}
	where $\lambda_{+}=e^{K}+e^{-K}$, and $\lambda_{-}=e^{K}-e^{-K}$. Let $\lambda_{+}(n)$ and $\lambda_{-}(n)$ represent the eigenvalues of transfer matrixes corresponding to Hamiltonians $\mathcal{H}_{\lambda=2^{n}}$ with coupling constants $K_{n}$, then		\begin{align}\label{4.19}
		e^{-2K_{n}}=\frac{2}{e^{K_{n-1}}+e^{-K_{n-1}}}.
	\end{align}
	Obviously,
	\begin{align*}
		\mathcal{G}_{\lambda=2^{n}}(x_{1},x_{2})&=\left (\frac{\lambda_{-}(n)}{\lambda_{+}(n)}\right)^{|x_{2}-x_{1}|}=\left (\frac{e^{K_{n}}-e^{-K_{n}}}{e^{K_{n}}+e^{-K_{n}}}\right)^{|x_{2}-x_{1}|},
	\end{align*}
	and
	\begin{align*}
		\mathcal{G}(2^{n}x_{1},2^{n}x_{2})&=\left (\frac{\lambda_{-}}{\lambda_{+}}\right)^{2^{n}|x_{2}-x_{1}|}=\left (\frac{e^{K}-e^{-K}}{e^{K}+e^{-K}}\right)^{2^{n}|x_{2}-x_{1}|}.
	\end{align*}
	It is clear that the second term $\mathcal{G}(2^{n}x_{1},2^{n}x_{2})$ in~\eqref{4.18}~ tends to 0 exponentially as $n\rightarrow \infty$. For the first term $\mathcal{G}_{\lambda=2^{n}}(x_{1},x_{2})$, because $K\geq 0$, then for all $n\geq 1$, from~\eqref{4.19}~
	\begin{align}
		0 &\leq K_{n}=\frac{1}{2}\log(\frac{e^{K_{n-1}}+e^{-K_{n-1}}}{2})\leq \frac{K_{n-1}}{2}+\frac{1}{2}\log(\frac{1+e^{-2K_{n-1}}}{2})\leq \frac{K_{n-1}}{2} \leq \frac{K_{n-2}}{2^{2}}\cdots\leq \frac{K}{2^{n}}.
	\end{align}
	Furthermore,
	\begin{align*}
		0 \leq  &\mathcal{G}_{\lambda=2^{n}}(x_{1},x_{2})=\left (\frac{e^{K_{n}}-e^{-K_{n}}}{e^{K_{n}}+e^{-K_{n}}}\right)^{|x_{2}-x_{1}|}
		\leq\left (\frac{e^{K_{n}}-e^{-K_{n}}}{2}\right)^{|x_{2}-x_{1}|}=e^{-K_{n}|x_{2}-x_{1}|}\left(\frac{e^{2K_{n}}-1}{2}\right)^{|x_{2}-x_{1}|}\\
		\leq& e^{2K}K_{n}^{|x_{2}-x_{1}|}
		\leq e^{2K}\frac{K^{|x_{2}-x_{1}|}}{2^{n{|x_{2}-x_{1}|}}}\ \rightarrow \ 0, \ as \ n\ \rightarrow \ \infty.
	\end{align*}
	Consequently, the exponential decay of the first term is obtained. Since both of them  decay to 0 exponentially, so dose their difference.
\end{proof}
\subsection{The "coarse-graining" procedure is consistent with the  observable \emph{RG} equations. }\label{sec:4.2}
The results in last section turn out that the remainder terms $\mathcal{S}_{n}(x_{1}, x_{2})$ decay to 0 exponentially in the frame of correlation function. As mentioned in Remark \ref{re3.2}, the behavior of two point correlation functions are similar to two point observables. Now, we attempt to reconstruct the \emph{RG} equation~\eqref{4.18}~by observables.

In the case of two point observables, the \emph{RG} equation can be written as
\begin{align}\label{4.21}
	\mathcal{S}^{fg}_{\lambda}(x_{1}, x_{2})-Z^{-1}(\lambda)\mathcal{S}^{fg}(\lambda x_{1}, \lambda x_{2})=\mathcal{O}_{\lambda}(x_{1}, x_{2}),
\end{align}
According to~\eqref{2.13}~, $\mathcal{S}^{f}$  can be divided into two parts $\hat{\mathcal{S}}^{f}$ and $\tilde{\mathcal{S}}^{f}$. Hence, the \emph{RG} equation corresponding to  the observables will be respectively written as
\begin{subequations}
	\begin{align}\label{eq:14}
		\hat{\mathcal{S}}^{fg}_{\lambda}(x_{1}, x_{2})-Z^{-1}(\lambda)\hat{\mathcal{S}}^{fg}(\lambda x_{1}, \lambda x_{2})=\hat{\mathcal{O}}_{\lambda}(x_{1}, x_{2}),\\
		\tilde{\mathcal{S}}^{fg}_{\lambda}(x_{1}, x_{2})-Z^{-1}(\lambda)\tilde{\mathcal{S}}^{fg}(\lambda x_{1}, \lambda x_{2})=\tilde{\mathcal{O}}_{\lambda}(x_{1}, x_{2}).
	\end{align}
\end{subequations}
We will adopt the similar statement as in correlation function.
Let factor $Z(\lambda)=1$, so for any $x_{1},x_{2}$ 
\begin{subequations}\label{eq:15}
	\begin{align}
		\mathcal{S}^{fg}_{\lambda=2}(x_{1}, x_{2})-\mathcal{S}^{fg}(2 x_{1}, 2 x_{2})=\mathcal{O}_{0}(x_{1}, x_{2}),\\
		\hat{\mathcal{S}}^{fg}_{\lambda=2}(x_{1}, x_{2})-\hat{\mathcal{S}}^{fg}(2 x_{1}, 2 x_{2})=\hat{\mathcal{O}}_{0}(x_{1}, x_{2}),\\
		\tilde{\mathcal{S}}^{fg}_{\lambda=2}(x_{1}, x_{2})-\tilde{\mathcal{S}}^{fg}(2 x_{1}, 2 x_{2})=\tilde{\mathcal{O}}_{0}(x_{1}, x_{2}).
	\end{align}
\end{subequations}
Then for any  given $n\in\mathds{N}$, iterating equation (\ref{eq:15}), and due to the arbitrariness of $x_{1},x_{2}$, the remainder terms $\tilde{\mathcal{O}}_{n}(x_{1}, x_{2})$ and $\hat{\mathcal{O}}_{n}(x_{1}, x_{2})$ are obtained and finally 
$\mathcal{O}_{n}(x_{1}, x_{2})=\hat{\mathcal{O}}_{n}(x_{1}, x_{2})-\tilde{\mathcal{O}}_{n}(x_{1}, x_{2})$.
\begin{theorem}\label{thm:4.2}
	The assumption is same as Theorem \ref{thm:4.1},  $\mathcal{S}^{fg}_{\lambda}(x_{1}, x_{2})$ are the two point observables corresponding to $\mathcal{H}_{\lambda}$. Then the remainder terms $\hat{\mathcal{O}}_{n}(x_{1}, x_{2})$ and $\tilde{\mathcal{O}}_{n}(x_{1}, x_{2})$ both decay to 0 exponentially.
\end{theorem}
\begin{proof}
	From theorem \ref{thm:3.1}, for any $x_{1},x_{2}\in\mathds{Z}$
	\begin{align*}
		\hat{\mathcal{S}}^{fg}=\frac{A}{4}+\frac{B}{4} \left (\frac{\lambda_{-}}{\lambda_{+}}\right)^{|x_{2}-x_{1}|}, \tilde{\mathcal{S}}^{fg}=\left(\frac{C^{4}_{12}C^{2}_{12}}{4}+\frac{\bar{\vartriangle}}{4}\left (\frac{\lambda_{-}}{\lambda_{+}}\right)^{|x_{2}-x_{1}|}\right)\log\left(\frac{C^{4}_{12}C^{2}_{12}}{4}+\frac{\bar{\vartriangle}}{4}\left (\frac{\lambda_{-}}{\lambda_{+}}\right)^{|x_{2}-x_{1}|}\right),
	\end{align*}
	where $\lambda_{+}=e^{K}+e^{-K}$ and $\lambda_{-}=e^{K}-e^{-K}$. Then
	\begin{align*}			\hat{\mathcal{O}}_{n}(x_{1}, x_{2})
		&=\hat{\mathcal{S}}^{fg}_{\lambda=2^{n}}(x_{1}, x_{2})-\hat{\mathcal{S}}^{fg}(2^{n} x_{1}, 2^{n} x_{2})=\frac{A}{4}+\frac{B}{4} \left (\frac{\lambda_{-}(n)}{\lambda_{+}(n)}\right)^{|x_{2}-x_{1}|}-\frac{A}{4}-\frac{B}{4} \left (\frac{\lambda_{-}}{\lambda_{+}}\right)^{2^{n}|x_{2}-x_{1}|}\\
		&=\frac{B}{4} \left \{\left(\frac{\lambda_{-}(n)}{\lambda_{+}(n)}\right)^{|x_{2}-x_{1}|}- \left (\frac{\lambda_{-}}{\lambda_{+}}\right)^{2^{n}|x_{2}-x_{1}|}\right\},
	\end{align*}
	and according to the results we get in Theorem \ref{thm:4.1}
	\begin{align*}
		0 \leq K_{n}\leq \frac{K_{n-1}}{2} \leq \frac{K_{n-2}}{2^{2}}\cdots\leq \frac{K}{2^{n}},
	\end{align*}
	and the estimate
	\begin{align*}
		\left (\frac{\lambda_{-}(n)}{\lambda_{+}(n)}\right)^{|x_{2}-x_{1}|}
		=&\left (\frac{e^{K_{n}}-e^{-K_{n}}}{e^{K_{n}}+e^{-K_{n}}}\right)^{|x_{2}-x_{1}|}
		\leq e^{2K}\frac{K^{|x_{2}-x_{1}|}}{2^{n{|x_{2}-x_{1}|}}}\ \rightarrow \ 0, \ as \ n\ \rightarrow \ \infty.
	\end{align*}
	It is easy to conclude that the remainder terms $\hat{\mathcal{O}}_{n}(x_{1}, x_{2})$ decay to 0 exponentially.	
	But for $\tilde{\mathcal{O}}_{n}(x_{1} x_{2})$
	\begin{align*}
		\tilde{\mathcal{O}}_{n}(x_{1} x_{2})&=\tilde{\mathcal{S}}^{fg}_{\lambda=2^{n}}(x_{1}, x_{2})-\tilde{\mathcal{S}}^{fg}(2^{n} x_{1}, 2^{n} x_{2})\\
		&=\left(\frac{C^{4}_{12}C^{2}_{12}}{4}+\frac{\bar{\vartriangle}}{4}\left (\frac{\lambda_{-}(n)}{\lambda_{+}(n)}\right)^{|x_{2}-x_{1}|}\right)\log\left(\frac{C^{4}_{12}C^{2}_{12}}{4}+\frac{\bar{\vartriangle}}{4}\left (\frac{\lambda_{-}(n)}{\lambda_{+}(n)}\right)^{|x_{2}-x_{1}|}\right)\\
		&-\left(\frac{C^{4}_{12}C^{2}_{12}}{4}+\frac{\bar{\vartriangle}}{4}\left (\frac{\lambda_{-}}{\lambda_{+}}\right)^{2^{n}|x_{2}-x_{1}|}\right)\log\left(\frac{C^{4}_{12}C^{2}_{12}}{4}+\frac{\bar{\vartriangle}}{4}\left (\frac{\lambda_{-}}{\lambda_{+}}\right)^{2^{n}|x_{2}-x_{1}|}\right)\\
		&=\frac{C^{4}_{12}C^{2}_{12}}{4}\log\frac{\frac{C^{4}_{12}C^{2}_{12}}{4}+\frac{\bar{\vartriangle}}{4}\left (\frac{\lambda_{-}(n)}{\lambda_{+}(n)}\right)^{|x_{2}-x_{1}|}}{\frac{C^{4}_{12}C^{2}_{12}}{4}+\frac{\bar{\vartriangle}}{4}\left (\frac{\lambda_{-}}{\lambda_{+}}\right)^{2^{n}|x_{2}-x_{1}|}}+\frac{\bar{\vartriangle}}{4}\left (\frac{\lambda_{-}(n)}{\lambda_{+}(n)}\right)^{|x_{2}-x_{1}|}\log\left(\frac{C^{4}_{12}C^{2}_{12}}{4}+\frac{\bar{\vartriangle}}{4}\left (\frac{\lambda_{-}(n)}{\lambda_{+}(n)}\right)^{|x_{2}-x_{1}|}\right)\\
		&-\frac{\bar{\vartriangle}}{4}\left (\frac{\lambda_{-}}{\lambda_{+}}\right)^{2^{n}|x_{2}-x_{1}|}\log\left(\frac{C^{4}_{12}C^{2}_{12}}{4}+\frac{\bar{\vartriangle}}{4}\left (\frac{\lambda_{-}}{\lambda_{+}}\right)^{2^{n}|x_{2}-x_{1}|}\right):= S_{1}+S_{2}+S_{3},
	\end{align*}
	where $S_{2}$ and $S_{3}$ satisfy the exponential decay obviously. For $S_{1}$ 
	\begin{align*}		S_{1}=\frac{C^{4}_{12}C^{2}_{12}}{4}\log\frac{\frac{C^{4}_{12}C^{2}_{12}}{4}+\frac{\bar{\vartriangle}}{4}\left (\frac{\lambda_{-}(n)}{\lambda_{+}(n)}\right)^{|x_{2}-x_{1}|}}{\frac{C^{4}_{12}C^{2}_{12}}{4}+\frac{\bar{\vartriangle}}{4}\left (\frac{\lambda_{-}}{\lambda_{+}}\right)^{2^{n}|x_{2}-x_{1}|}}=&\frac{C^{4}_{12}C^{2}_{12}}{4}\log\left(1+\frac{\bar{\vartriangle}}{4}\frac{\left (\frac{\lambda_{-}(n)}{\lambda_{+}(n)}\right)^{|x_{2}-x_{1}|}-\left (\frac{\lambda_{-}}{\lambda_{+}}\right)^{2^{n}|x_{2}-x_{1}|}}{\frac{C^{4}_{12}C^{2}_{12}}{4}+\frac{\bar{\vartriangle}}{4}\left (\frac{\lambda_{-}}{\lambda_{+}}\right)^{2^{n}|x_{2}-x_{1}|}}\right)\\
		\leq&\frac{C^{4}_{12}C^{2}_{12}}{4}\frac{\bar{\vartriangle}}{4}\frac{\left (\frac{\lambda_{-}(n)}{\lambda_{+}(n)}\right)^{|x_{2}-x_{1}|}-\left (\frac{\lambda_{-}}{\lambda_{+}}\right)^{2^{n}|x_{2}-x_{1}|}}{\frac{C^{4}_{12}C^{2}_{12}}{4}+\frac{\bar{\vartriangle}}{4}\left (\frac{\lambda_{-}}{\lambda_{+}}\right)^{2^{n}|x_{2}-x_{1}|}},
	\end{align*}
	for $n$ sufficiently large. Clearly, the right hand side of the inequality tends to 0 exponentially as $n\rightarrow \infty$. Since the remainder terms $\hat{\mathcal{O}}_{n}(x_{1}, x_{2})$ and $\tilde{\mathcal{O}}_{n}(x_{1}, x_{2})$ both tend to 0 exponentially, the  exponential decay of $\mathcal{O}_{n}(x_{1}, x_{2})$ is obtained.
\end{proof}
\section{The observables of stochastic dynamic}\label{sec:5}
Now, we begin to consider the Cauchy problem of the discrete stochastic dynamic~\eqref{5.11}~with initial distribution $\mathds{P}_{0}$ and system parameter $K_{\lambda(t)}$
\begin{equation}\label{5.11}
	\begin{cases}
		\phi(x,t+1)&=\mathcal{P}(\phi(x,t),\xi(x,t))\\
		\phi(\cdot,0)&=\mathbb{P}_{0}.
	\end{cases}
\end{equation}
where $\mathbb{P}_{0}$ is the infinite-volume Gibbs measure associated to the Ising Hamiltonian with nearest  neighbor interaction. Specifically, for any finite subset $\Lambda\subset \mathds{Z}$,  we decompose the measure $\mathbb{P}_{0}$ into the marginal $\mathbb{P}_{0}(\phi(x,0)=\sigma_{x},x\in \Lambda^{c})$ and  the conditional measure $\mathbb{P}_{0}(\phi(x,0)=\sigma_{x},x\in \Lambda|\phi(y,0)=\sigma_{y},y\in \Lambda^{c})$. Furthermore,  the conditional measure satisfies 
\begin{align}\label{eq:5.2}
	\mathbb{P}_{0}(\phi(x,0)=\sigma_{x},x\in \Lambda|\phi(y,0)=\sigma_{y},y\in \Lambda^{c})=\frac{1}{Z_{\partial_{\Lambda}}}\exp\{K\sum\limits_{x\in\Lambda:x \sim y}\sigma_{x}\sigma_{y}\},
\end{align}
which is also called Dobrushin Lanford Ruelle(DLR) equations; see \cite{Menzdocument2014} for details.

\subsection{The estimate of two point  observables. }\label{sec:5.1}
This section we devote to investigating the behavior of two point observables induced by the stochastic difference equation~\eqref{5.11}~. The trick is to replace the coupling constant $K$ by $K_{\lambda(t)}$ which is associsted with time $t$ and the dilatation factor $\lambda(t)$, and satisfies
\begin{enumerate}
	\item  $\lambda(t) \rightarrow \infty$, as $t\rightarrow \infty$;
	\item $K_{\lambda(t)}\rightarrow K^{*}$, as $\lambda(t)\rightarrow \infty$;
	\item 
	$\mathcal{S}^{fg}_{\lambda(t)}(x_{1}, x_{2})-Z^{-1}(\lambda(t))\mathcal{S}^{fg}(\lambda(t) x_{1}, \lambda(t) x_{2})=\mathcal{O}_{\lambda(t)}(x_{1}, x_{2})\rightarrow 0$, as $t\rightarrow \infty.$
\end{enumerate}
Generally, in the vicinity of fixed point, the two point observables convergences as $K_{\lambda(t)}\rightarrow0$, that is to say, $K^{*}=0$ is a stable fixed point.
Now, we begin to consider the Cauchy problem of the discrete stochastic dynamic~\eqref{5.11}~with initial distribution ~\eqref{eq:5.2}~ and system parameter $K_{\lambda(t)}$
\begin{equation*}
	\begin{cases}
		\phi(x,t+1)&=\mathcal{P}(\phi(x,t),\xi(x,t)),\\
		\phi(\cdot,0)&=\mathbb{P}_{0}.
	\end{cases}
\end{equation*}
Obviously, thanks to the structure of iteration, for any time $s< t$ and spatial coordinates $x_{i}, x_{j}\in\mathbb{Z}$, $\phi(x_{i},s)$ is independent to $\xi(x_{j},t)$. Without loss of generality, we assume that $x_{1}, x_{2}\in\mathbb{Z}$ satisfies $\left| x_{2}-x_{1}\right| \geq 3$. Let vectors 
$\left(\sigma_{1} ,  \sigma_{2} ,\sigma_{3} , \sigma_{4} , \sigma_{5}, \sigma_{6}  \right) $ represent the six dimensional spin variables, where $\sigma_{i}=\pm 1, i=1,\dots,6$  are spins. For any $t>0$, the probability of each spin variable is
\begin{align*}
	\mathbb{P}\big(\phi(x_{1}-1,t)=\sigma_{1},\phi(x_{1},t)=\sigma_{2},\dots,\phi(x_{2},t)=\sigma_{5},\phi(x_{2}+1,t)= \sigma_{6}\big).
\end{align*}
For convenience, denote that
\begin{equation*}
	\left\{
	\begin{aligned}
		&A_{11}(t)={ K_{\lambda(t)}(\phi(x_{1}-1,t)+\phi(x_{1}+1,t))+(1-\gamma)\phi(x_{1},t)+\xi(x_{1},t)\geq 0},\\
		&A_{12}(t)={ K_{\lambda(t)}(\phi(x_{2}-1,t)+\phi(x_{2}+1,t))+(1-\gamma)\phi(x_{2},t)+\xi(x_{2},t)\geq 0},\\
		&A_{21}(t)={ K_{\lambda(t)}(\phi(x_{1}-1,t)+\phi(x_{1}+1,t))+(1-\gamma)\phi(x_{1},t)+\xi(x_{1},t)\leq 0},\\
		&A_{22}(t)={ K_{\lambda(t)}(\phi(x_{2}-1,t)+\phi(x_{2}+1,t))+(1-\gamma)\phi(x_{2},t)+\xi(x_{2},t)\geq 0},
	\end{aligned}
	\right.
\end{equation*}
and
\begin{equation*}
	\left\{
	\begin{aligned}
		&A_{31}(t)={ K_{\lambda(t)}(\phi(x_{1}-1,t)+\phi(x_{1}+1,t))+(1-\gamma)\phi(x_{1},t)+\xi(x_{1},t)\geq 0},\\
		&A_{32}(t)={ K_{\lambda(t)}(\phi(x_{2}-1,t)+\phi(x_{2}+1,t))+(1-\gamma)\phi(x_{2},t)+\xi(x_{2},t)\leq 0},\\
		&A_{41}(t)={ K_{\lambda(t)}(\phi(x_{1}-1,t)+\phi(x_{1}+1,t))+(1-\gamma)\phi(x_{1},t)+\xi(x_{1},t)\leq 0},\\
		&A_{42}(t)={ K_{\lambda(t)}(\phi(x_{2}-1,t)+\phi(x_{2}+1,t))+(1-\gamma)\phi(x_{2},t)+\xi(x_{2},t)\leq 0}.
	\end{aligned}
	\right.
\end{equation*}
Undergoing the structure rearrangement of the six dimensional spin variables in a binary sort as follows
\begin{align}\label{5.1.2}
	&\mathbb{P}_{t}=\left(                 
	\begin{array}{cccccc}
		\\
		P^{1}_{t}\\
		P^{2}_{t}\\	
		P^{3}_{t}\\
		P^{4}_{t}\\
		P^{5}_{t}\\
		P^{6}_{t}\\	
		P^{7}_{t}\\
		P^{8}_{t}\\
		\vdots \\
		\vdots\\
		P^{64}_{t}
	\end{array}
	\right)\leftrightarrow            
	\left(                 
	\begin{array}{rrrrrrrr}
		\sigma_{1}&\sigma_{2}&\sigma_{3}&\sigma_{4}&\sigma_{5}&\sigma_{6}\\
		1 &1 &1 &1 &1 &1\\
		-1&1 &1 &1 &1 &1\\	
		1 &-1 &1 &1 &1 &1\\
		-1&-1 &1 &1 &1 &1\\
		1 &1 &-1 &1 &1 &1\\
		-1&1 &-1 &1 &1 &1\\	
		1 &-1 &-1 &1 &1 &1\\
		-1&-1 &-1 &1 &1 &1\\
		\qquad &\vdots &\qquad &\vdots\\
		\qquad &\vdots &\qquad &\vdots\\
		-1&-1 &-1 &-1 &-1 &-1
	\end{array}
	\right),            
\end{align}	
where $\{P_{t}^{i}\}_{i=1,\cdots,64}$ are the probability values corresponding to these six dimensional spins. By definition
\begin{align}\label{5.1.3}
	\mathbb{P}(\phi  (x_{1},t+1)=&1  ,\phi(x_{2},t+1)= 1)\nonumber\\
	=&
	\mathbb{P}\big( A_{11}A_{12}\big)=\mathbb{P}\big(2K_{\lambda(t)}+1-\gamma+\xi(x_{1},t)\geq 0,2K_{\lambda(t)}+1-\gamma+\xi(x_{2},t)\geq 0 \big) \cdot P^{1}_{t}\nonumber\\
	&\vdots\nonumber\\
	+&\mathbb{P}\big(-2K_{\lambda(t)}-(1-\gamma)+\xi(x_{1},t)\geq 0,-2K_{\lambda(t)}-(1-\gamma)+\xi(x_{2},t)\geq 0 \big) \cdot P^{64}_{t}.
\end{align}
Owing to the independence of Gaussian random variables, we know that  
\begin{align*}
	\mathbb{P}\big(x+\xi(x_{1},t)\geq 0,y+\xi(x_{2},t)\geq 0 \big)&=\mathbb{P}\big(x+\xi(x_{1},t)\geq 0\big)\cdot \mathbb{P}\big(y+\xi(x_{2},t)\geq 0 \big)=\Phi(x)\Phi(y),
\end{align*}
and 
\begin{align*}
	\mathbb{P}\big(x+\xi(x_{1},t)\leq 0,y+\xi(x_{2},t)\geq 0 \big)&=\Phi(-x)\Phi(y),\\
	\mathbb{P}\big(x+\xi(x_{1},t)\geq 0,y+\xi(x_{2},t)\leq 0 \big)&=\Phi(x)\Phi(-y),\\
	\mathbb{P}\big(x+\xi(x_{1},t)\leq 0,y+\xi(x_{2},t)\leq 0 \big)&=\Phi(-x)\Phi(-y),
\end{align*}
where $\Phi(x)$ is the Gauss integration
\begin{align*}
	\Phi(x)=\frac{1}{\sqrt{2\pi }}\int_{-\infty}^{x}e^{-\frac{z^{2}}{2}}dz.
\end{align*}
Furthermore, the new spin variables  $\left(S_{2} ,  S_{5}\right)$ are introduced
\begin{align*}
	\left(                 
	\begin{array}{l}
		\qquad \qquad   \ \text{New Spin}\\
		\phi(x_{1},t+1)= 1, \phi(x_{2},t+1)= 1\\  
		\phi(x_{1},t+1)=-1, \phi(x_{2},t+1)= 1\\  
		\phi(x_{1},t+1)= 1, \phi(x_{2},t+1)=-1\\
		\phi(x_{1},t+1)=-1, \phi(x_{2},t+1)=-1
	\end{array}
	\right)            
	=\left(                 
	\begin{array}{rr}
		S_{2} & S_{5} \\  
		1&1 \\  
		-1&1\\
		1&-1\\
		-1&-1
	\end{array}
	\right).           
\end{align*}	 
Similarly, probability of the other three spin variables can also be calculated by equation~\eqref{5.1.3}~, then  
\begin{align}\label{5.1.4}
	\mathbb{A}_{t+1}  := \left(                 
	\begin{array}{l}
		\mathbb{P}\left(\phi(x_{1},t+1)= 1,\phi(x_{2},t+1)= 1\right)  \\  
		\mathbb{P}\left(\phi(x_{1},t+1)=-1,\phi(x_{2},t+1)= 1\right) \\  
		\mathbb{P}\left(\phi(x_{1},t+1)= 1,\phi(x_{2},t+1)=-1\right)\\
		\mathbb{P}\left(\phi(x_{1},t+1)=-1,\phi(x_{2},t+1)=-1\right)
	\end{array}
	\right)
	=\left(                 
	\begin{array}{l}
		\mathbb{P}\left(A_{11}A_{12}\right)  \\  
		\mathbb{P}\left(A_{21}A_{22}\right) \\  
		\mathbb{P}\left(A_{31}A_{32}\right)\\
		\mathbb{P}\left(A_{41}A_{42}\right)
	\end{array}
	\right)            
	:= \Theta_{K_{\lambda(t)}}\mathbb{P}_{t},
\end{align}	
where $\Theta_{K_{\lambda(t)}}$ is a 4$\times$64 transition probability matrix, and $\mathbb{A}_{t+1}$ is the joint distribution of two point new spin variables at time $t+1$,  $\mathbb{P}_{t}$ is the joint distribution of original six dimensional spins at time $t$.
According to the expansion \eqref{5.1.3}~, the elements of $\Theta_{K_{\lambda(t)}}$ can be expressed explicitly
\begin{align}\label{5.1.5}
	\Theta_{K_{\lambda(t)}}=
	\left(                 
	\begin{array}{lll}
		\dots &   \Phi(K_{\lambda(t)}(\sigma_{1}+\sigma_{3})+(1-\gamma)\sigma_{2})\Phi(K_{\lambda(t)}(\sigma_{4}+\sigma_{6})+(1-\gamma)\sigma_{5}) & \dots\\
		\dots& \Phi(-K_{\lambda(t)}(\sigma_{1}+\sigma_{3})-(1-\gamma)\sigma_{2})\Phi(K_{\lambda(t)}(\sigma_{4}+\sigma_{6})+(1-\gamma)\sigma_{5}) &\dots\\
		\dots & \Phi(K_{\lambda(t)}(\sigma_{1}+\sigma_{3})+(1-\gamma)\sigma_{2})\Phi(-K_{\lambda(t)}(\sigma_{4}+\sigma_{6})-(1-\gamma)\sigma_{5}) & \dots\\
		\dots & \Phi(-K_{\lambda(t)}(\sigma_{1}+\sigma_{3})-(1-\gamma)\sigma_{2})\Phi(-K_{\lambda(t)}(\sigma_{4}+\sigma_{6})-(1-\gamma)\sigma_{5}) & \dots
	\end{array}
	\right) ,           
\end{align}	
and the values of spins $\left(\sigma_{1} , \sigma_{2} ,\sigma_{3} , \sigma_{4} , \sigma_{5}, \sigma_{6}  \right) $ in $i$-th  column of matrix $\Theta_{K_{\lambda(t)}}$  corresponds to the $i$-th row of  $\mathbb{P}_{t}$ in~\eqref{5.1.2}~. By the definition of observables, using symbols $a,b,c,d$ introduced in Theorem \ref{thm:3.1}
\begin{align}\label{5.1.6}
	\hat{\mathcal{S}}^{fg}_{\lambda(t)}(x_{1}, x_{2})=&\mathbb{E}_{K_{\lambda(t)}}\left[ f^{2}(\phi(x_{1},t+1))g^{2}(\phi(x_{2},t+1))\log\left(f^{2}(\phi(x_{1},t+1))g^{2}(\phi(x_{2},t+1))\right)\right]\nonumber\\
	=&a\mathbb{P}\left(A_{11}A_{12}\right) 
	+b\mathbb{P}\left(A_{21}A_{22}\right) +c\mathbb{P}\left(A_{31}A_{32}\right) +d\mathbb{P}\left(A_{41}A_{42}\right), 
\end{align}
and
$$\tilde{\mathcal{S}}^{fg}_{\lambda(t)}(x_{1}, x_{2})=\|fg\|^{2}_{L^{2}(K_{\lambda(t)})}\log (\|fg\|^{2}_{L^{2}(K_{\lambda(t)})}),$$
where
\begin{align}\label{5.1.7}
	\|fg\|^{2}_{L^{2}(K_{\lambda(t)})}:=& \mathbb{E}_{K_{\lambda(t)}}\left[ f^{2}(\phi(x_{1},t+1))g^{2}(\phi(x_{2},t+1))\right]\nonumber\\
	=&\tilde{a}\mathbb{P}\left(A_{11}A_{12}\right) 	+\tilde{b}\mathbb{P}\left(A_{21}A_{22}\right) +\tilde{c}\mathbb{P}\left(A_{31}A_{32}\right) +\tilde{d}\mathbb{P}\left(A_{41}A_{42}\right), 
\end{align}
and
$$\tilde{a}=f^{2}(1)g^{2}(1),\tilde{b}=f^{2}(-1)g^{2}(1),\tilde{c}=f^{2}(1)g^{2}(-1),\tilde{d}=f^{2}(-1)g^{2}(-1).$$
On the one hand, owing to~\eqref{5.1.4}~and ~\eqref{5.1.6}~
\begin{align}\label{5.1.8}
	\hat{E}_{t+1}:= \hat{\mathcal{S}}^{fg}_{\lambda(t)}(x_{1}, x_{2})
	=&\left(                 
	\begin{array}{llll}   
		a & b & c & d   
	\end{array}
	\right)
	\mathbb{A}_{t+1}    =\left(                 
	\begin{array}{llll}   
		a & b & c & d   
	\end{array}
	\right)\Theta_{K_{\lambda(t)}}\mathbb{P}_{t}   .          
\end{align}	
On the other hand, according to equation~\eqref{5.1.6}~, the value $\hat{E}_{t+1}$  is totally determined by spins $\sigma_{2}$ , $\sigma_{5}$ at $x_{1},x_{2}$ in~\eqref{5.1.2}~but replaced time $t$ by $t+1$, and is independent to spins at $x_{1}-1,x_{1}+1,x_{2}-1,x_{2}+1$, which means $\sigma_{1}$, $\sigma_{3}$ , $\sigma_{4}$ , $\sigma_{6}$. Hence
\begin{align*}
	\hat{E}_{t+1}
	=\left(                 
	\begin{array}{llllllllllllllllllll}   
		a & a & b & b & \dots & c & c & d & d  & a & a & b & b & \dots & c & c & d & d 
	\end{array}
	\right)\mathbb{P}_{t+1}.
\end{align*}	
Denote that 		
\begin{align*}
	\vec{A}:=\vec{A}^{0}
	=\left(                 
	\begin{array}{llllllllllllllllllll}   
		a & a & b & b & \dots & c & c & d & d  & a & a & b & b & \dots & c & c & d & d 
	\end{array}
	\right).
\end{align*}
Generally, let
\begin{align}\label{5.1.9}
	\vec{A}^{i}
	=\left(                 
	\begin{array}{llllllllllllllllllll}   
		a_{i} & a _{i}& b_{i} & b_{i} & \dots & c_{i} & c_{i} & d_{i} & d_{i}  & a_{i} & a_{i} & b_{i} & b_{i} & \dots & c_{i} & c_{i} & d_{i} & d_{i} 
	\end{array}
	\right), i\geq 1.
\end{align}	
that is to say
\begin{align}\label{5.1.10}
	\hat{E}_{t+1}
	=\vec{A}\cdot \mathbb{P}_{t+1},
\end{align}	
where $\vec{A}$ is a $1\times 64$ vector, but  includes only four different elements $a,b,c,d$. The relationship between positions of $a,b,c,d$ in  $\vec{A}$ and the values of spins  $\sigma_{2}$ , $\sigma_{5}$ in~\eqref{5.1.2}~are
\begin{align}\label{5.1.11}		a\leftrightarrow \left(\sigma_{2}=1 , \sigma_{5}=1\right), b\leftrightarrow \left(\sigma_{2}=-1 , \sigma_{5}=1\right), c\leftrightarrow \left(\sigma_{2}=1 , \sigma_{5}=-1\right), d\leftrightarrow \left(\sigma_{2}=-1 , \sigma_{5}=-1\right).
\end{align}
In order to get the estimate of observables, the following computation of Gaussian integration and transition probability matrix are necessary.
\begin{lemma}\label{lem:5.1}
	Let $\tilde{\Phi}$ be the matrix with following form
	\begin{align}\label{5.1.12}
		\tilde{\Phi}=
		\left(                 
		\begin{array}{llll}
			\Phi(x)\Phi(x) & \Phi(-x)\Phi(x) &  \Phi(x)\Phi(-x) & \Phi(-x)\Phi(-x)\\
			\Phi(-x)\Phi(x) & \Phi(x)\Phi(x) &  \Phi(-x)\Phi(-x) & \Phi(x)\Phi(-x)\\
			\Phi(x)\Phi(-x) & \Phi(-x)\Phi(-x) & \Phi(x)\Phi(x) & \Phi(-x)\Phi(x))\\
			\Phi(-x)\Phi(-x) & \Phi(x)\Phi(-x) & \Phi(-x)\Phi(x) & \Phi(x)\Phi(x)
		\end{array}
		\right),
	\end{align}	
	where $\Phi(x)$ is the Gauss integration defined as above, for $x\in \mathbb{R}$. Then we claim that the spectral radius of  $\tilde{\Phi}$ equals to 1, and algebraic multiplicity of the maximal eigenvalue is 1. Furthermore, the matrix sequence $\tilde{\Phi}^{s+1}$ converges as follows
	\begin{align*}
		\tilde{\Phi}^{s+1}\rightarrow\xi_{1}\xi_{1}^{T}
		=\dfrac{1}{4}\left(                 
		\begin{array}{llll}   
			1&1&1&1\\
			1&1&1&1\\
			1&1&1&1\\
			1&1&1&1
		\end{array}
		\right):=\frac{1}{4}\mathds{1}_{4}  \ as \ s\rightarrow \infty,
	\end{align*} 
	where $\xi_{1}$ is the eigenvector corresponding to the maximal eigenvalue and
	\begin{align*}
		\xi_{1}^{T}
		=\frac{1}{2}\left(                 
		\begin{array}{llll}   
			1&1&1&1
		\end{array}
		\right).
	\end{align*} 
\end{lemma}	
\begin{proof}
	Obviously, for Gaussian integration, we know that $\Phi(x)=1-\Phi(-x)$. Then the row sum of  $\tilde{\Phi}$ is 
	$$\Phi(x)\Phi(x) +\Phi(-x)\Phi(x) +  \Phi(x)\Phi(-x) + \Phi(-x)\Phi(-x)=(\Phi(x) +\Phi(-x))^{2}=1.$$
	Because $\forall x\in \mathbb{R}$, $\Phi(x)>0$, then entries of $\tilde{\Phi}$ are all positive. By Gersgorin Circle Theorem in linear algebra, the spectral radius of $\tilde{\Phi}$ less than or equal to 1, and 1 is a eigenvalue of $\tilde{\Phi}$ with unit eigenvector $\xi_{1}$. Let the other eigenvalues be $\lambda_{2},\lambda_{3},\lambda_{4}$ which corresponding to unit eigenvectors $\xi_{2},\xi_{3},\xi_{4}$. Through solving the equation
	$$\det\left(\tilde{\Phi}-\lambda I \right)=0,$$
	we can get explicitly
	$$\lambda_{1}=1,\lambda_{2}=\lambda_{3}=\Phi(x)-\Phi(-x),\lambda_{4}=(\Phi(x)-\Phi(-x))^{2}.$$ 
	Then $\tilde{\Phi}$ is diagonalizable and $O^{T}\tilde{\Phi}O=diag\{1,\lambda_{2},\lambda_{3},\lambda_{4}\}$, where  $O=(\xi_{1},\xi_{2},\xi_{3},\xi_{4})$ is an orthogonal matrix. Obviously, for $x\in \mathbb{R}$, $|\lambda_{2}|=|\lambda_{3}|=|\Phi(x)-\Phi(-x)|<1$ and $\lambda_{4}=(\Phi(x)-\Phi(-x))^{2}<1$. Hence, $\tilde{\Phi}$ satisfies		\begin{align*}
		\tilde{\Phi}^{s+1}\rightarrow &
		\left(                 
		\begin{array}{llll} 
			\xi_{1}&\xi_{2}&\xi_{3}&\xi_{4}
		\end{array}
		\right)
		diag\{1,0,0,0\}\left(                 
		\begin{array}{llll} 
			\xi_{1}&\xi_{2}&\xi_{3}&\xi_{4}
		\end{array}
		\right)^{T}
		=\xi_{1}\xi_{1}^{T}
		=\frac{1}{4}
		\mathds{1}_{4},
	\end{align*} 
	as $ s\rightarrow \infty$,	which completes the lemma \ref{lem:5.1}.
\end{proof}
\begin{lemma}\label{lem:5.2}
	Let $\Phi(x)$, $x\in \mathbb{R}$ be the Gaussian integration defined as above
	\begin{align*}
		\Phi(x)=\frac{1}{\sqrt{2\pi }}\int_{-\infty}^{x}e^{-\frac{z^{2}}{2}}dz.
	\end{align*}
	Consider function $h(x,y)=\Phi(x)\Phi(y)$, Then the following uniform estimate can be obtained 
	\begin{align*}
		|h(x_{1},x_{2})-h(y_{1},y_{2})|\leq  C(|x_{1}-y_{1}|+|x_{2}-y_{2}|).
	\end{align*} 
\end{lemma}	
\begin{proof}
	The proof is trivial.
\end{proof}
\begin{theorem}\label{th:5.3}
	Considering the stochastic dynamic~\eqref{5.11}~ with system parameter $K_{\lambda(t)}$ and initial distribution  $\mathbb{P}_{0}$. The two point observables $\mathcal{S}^{fg}_{\lambda(t)}(x_{1}, x_{2})$ and  $\mathcal{S}^{fg}(\lambda(t) x_{1}, \lambda(t) x_{2})$ satisfies
	\begin{align*}
		\mathcal{S}^{fg}_{\lambda(t)}(x_{1}, x_{2})&\rightarrow\frac{A}{4}-\frac{C^{4}_{12}C^{2}_{12}}{4}\log \left(\frac{C^{4}_{12}C^{2}_{12}}{4}\right),\\
		\mathcal{S}^{fg}(\lambda(t) x_{1}, \lambda(t) x_{2})&\rightarrow\frac{A}{4}-\frac{C^{4}_{12}C^{2}_{12}}{4}\log \left(\frac{C^{4}_{12}C^{2}_{12}}{4}\right),
	\end{align*}
	as $t\rightarrow \infty, \lambda(t)\rightarrow \infty$ and $K_{\lambda(t)}\rightarrow K^{*}=0$. 
\end{theorem}
\begin{proof}
	Let us start from the special case $K_{\lambda(t)}=0$ for  $t\geq0$ firstly, and denote that $\hat{E}_{t+1}$ as $\hat{E}^{0}_{t+1}$. Then according to ~\eqref{5.1.5}~, the transition probability matrix can be denoted by $\Phi^{0}$
	\begin{align*}
		\Phi^{0}=
		\left(                 
		\begin{array}{lll}
			\dots &   \Phi((1-\gamma)\sigma_{2})\Phi((1-\gamma)\sigma_{5}) & \dots\\
			\dots& \Phi(-(1-\gamma)\sigma_{2})\Phi((1-\gamma)\sigma_{5}) &\dots\\
			\dots & \Phi((1-\gamma)\sigma_{2})\Phi(-(1-\gamma)\sigma_{5}) & \dots\\
			\dots & \Phi(-(1-\gamma)\sigma_{2})\Phi(-(1-\gamma)\sigma_{5}) & \dots
		\end{array}
		\right).          
	\end{align*}	
	Noticed that the entries of $\Phi^{0}$ are determined by spins $\sigma_{2}$ and $\sigma_{5}$, then there are only four different columns in $\Phi^{0}$, even through it is a  $4\times 64$ matrix. The relationship between columns in $\Phi^{0}$ and spins $\sigma_{2},\sigma_{5}$  are
	\begin{align}\label{5.1.13} 
		\begin{array}{lllllll}
			\left(\sigma_{2}=1 , \sigma_{5}=1\right) \\
			\ \ \leftrightarrow (\Phi(1-\gamma)\Phi(1-\gamma) & \Phi(\gamma-1)\Phi(1-\gamma) &  \Phi(1-\gamma)\Phi(\gamma-1) & \Phi(\gamma-1)\Phi(\gamma-1))^{T},\\
			\left(\sigma_{2}=-1 , \sigma_{5}=1\right) \\
			\ \  \leftrightarrow (\Phi(\gamma-1)\Phi(1-\gamma) & \Phi(1-\gamma)\Phi(1-\gamma) &  \Phi(\gamma-1)\Phi(\gamma-1) & \Phi(1-\gamma)\Phi(\gamma-1))^{T},\\
			\left(\sigma_{2}=1 , \sigma_{5}=-1\right) \\ 
			\ \ \leftrightarrow (\Phi(1-\gamma)\Phi(\gamma-1) & \Phi(\gamma-1)\Phi(\gamma-1) & \Phi(1-\gamma)\Phi(1-\gamma) & \Phi(\gamma-1)\Phi(1-\gamma))^{T},\\
			\left(\sigma_{2}=-1 , \sigma_{5}=-1\right)\\ 
			\ \ \leftrightarrow (\Phi(\gamma-1)\Phi(\gamma-1) & \Phi(1-\gamma)\Phi(\gamma-1) & \Phi(\gamma-1)\Phi(1-\gamma) & \Phi(1-\gamma)\Phi(1-\gamma))^{T}.
		\end{array}
	\end{align}
	Combining ~\eqref{5.1.11}~and~\eqref{5.1.13}~, the columns in $\Phi^{0}$ and positions of $a,b,c,d$ in  $\vec{A}$ can be related by values of spins $\sigma_{2},\sigma_{5}$. More explicitly
	\begin{align}\label{4.12} 
		\begin{array}{lllllll}
			a \leftrightarrow \left(\sigma_{2}=1 , \sigma_{5}=1\right) \\
			\ \ \ \leftrightarrow (\Phi(1-\gamma)\Phi(1-\gamma) & \Phi(\gamma-1)\Phi(1-\gamma) &  \Phi(1-\gamma)\Phi(\gamma-1) & \Phi(\gamma-1)\Phi(\gamma-1))^{T},\\
			b \leftrightarrow \left(\sigma_{2}=-1 , \sigma_{5}=1\right) \\
			\ \ \ \leftrightarrow (\Phi(\gamma-1)\Phi(1-\gamma) & \Phi(1-\gamma)\Phi(1-\gamma) &  \Phi(\gamma-1)\Phi(\gamma-1) & \Phi(1-\gamma)\Phi(\gamma-1))^{T},\\
			c\leftrightarrow \left(\sigma_{2}=1 , \sigma_{5}=-1\right) \\ \ \ \ \leftrightarrow (\Phi(1-\gamma)\Phi(\gamma-1) & \Phi(\gamma-1)\Phi(\gamma-1) & \Phi(1-\gamma)\Phi(1-\gamma) & \Phi(\gamma-1)\Phi(1-\gamma))^{T},\\
			d\leftrightarrow \left(\sigma_{2}=-1 , \sigma_{5}=-1\right)\\ \ \ \ \leftrightarrow (\Phi(\gamma-1)\Phi(\gamma-1) & \Phi(1-\gamma)\Phi(\gamma-1) & \Phi(\gamma-1)\Phi(1-\gamma) & \Phi(1-\gamma)\Phi(1-\gamma))^{T}.
		\end{array}
	\end{align}
	Based on this crucial relationship, on the one hand, according to~\eqref{5.1.8}~ and symbols introduced in~\eqref{5.1.9}~,
	\begin{align*}
		\hat{E}^{0}_{t+1}
		=&\left(                 
		\begin{array}{llll}   
			a & b & c & d   
		\end{array}
		\right)\Phi_{0}\mathbb{P}_{t}=\vec{A}^{1}\mathbb{P}_{t}.
	\end{align*}	
	The elements in $\vec{A}^{1}$ satisfy $(a_{1}\  b_{1}\ c_{1}\  d_{1})=(a\  b\ c\  d) \hat{\Phi}$,
	where $\hat{\Phi}$ is the so-called \emph{ tranfer determined matrix} with the same form as ~\eqref{5.1.12}~, that is
	\begin{align*}
		\hat{\Phi}=
		\left(                 
		\begin{array}{llll}
			\Phi(1-\gamma)\Phi(1-\gamma) & \Phi(\gamma-1)\Phi(1-\gamma) &  \Phi(1-\gamma)\Phi(\gamma-1) & \Phi(\gamma-1)\Phi(\gamma-1)\\
			\Phi(\gamma-1)\Phi(1-\gamma) & \Phi(1-\gamma)\Phi(1-\gamma) &  \Phi(\gamma-1)\Phi(\gamma-1) & \Phi(1-\gamma)\Phi(\gamma-1)\\
			\Phi(1-\gamma)\Phi(\gamma-1) & \Phi(\gamma-1)\Phi(\gamma-1) & \Phi(1-\gamma)\Phi(1-\gamma) & \Phi(\gamma-1)\Phi(1-\gamma))\\
			\Phi(\gamma-1)\Phi(\gamma-1) & \Phi(1-\gamma)\Phi(\gamma-1) & \Phi(\gamma-1)\Phi(1-\gamma) & \Phi(1-\gamma)\Phi(1-\gamma)
		\end{array}
		\right).
	\end{align*}
	On the other hand, according to~\eqref{5.1.10}~, we know $\hat{E}^{0}_{t+1}=\vec{A}^{0}\mathbb{P}_{t+1}$,	
	which leads to $\hat{E}^{0}_{t+1}		=\vec{A}^{0}\mathbb{P}_{t+1}=\vec{A}^{1}\mathbb{P}_{t}$ immediately. Noticed that the discussion above works for any time $t$, then we can get following iteration 
	\begin{align}\label{4.14}
		\hat{E}^{0}_{t+1}
		=\vec{A}^{0}\mathbb{P}_{t+1}
		=\vec{A}^{1}\mathbb{P}_{t}= \cdots =\vec{A}^{t+1}\mathbb{P}_{0},
	\end{align}	
	where the parameters in $\vec{A}_{i},1\leq i\leq t+1$ are determined by
	\begin{align*}
		\left(                 
		\begin{array}{llll}   
			a_{t+1} & b_{t+1} & c_{t+1} & d_{t+1}   
		\end{array}
		\right)
		=&\left(                 
		\begin{array}{llll}   
			a_{t} & b_{t} & c_{t} & d_{t}   
		\end{array}
		\right)\hat{\Phi}=
		\cdots 
		=\left(                 
		\begin{array}{llll}   
			a & b & c & d   
		\end{array}
		\right)\hat{\Phi}^{t+1}.
	\end{align*}
	Owing to Lemma \ref{lem:5.1}, as $t\rightarrow \infty$		\begin{align*}
		\left(                 
		\begin{array}{llll}   
			a_{t+1} & b_{t+1} & c_{t+1} & d_{t+1}   
		\end{array}
		\right)
		&\rightarrow\frac{1}{4}\left(                 
		\begin{array}{llll}   
			a & b & c & d   
		\end{array}
		\right)\mathds{1}_{4}=\frac{a+b+c+d}{4}\left(                 
		\begin{array}{llll}   
			1 & 1 & 1 & 1   
		\end{array}
		\right).
	\end{align*}
	Because the initial data $\mathbb{P}_{0}$ is a probability distribution, then as $t\rightarrow \infty$, 
	\begin{align}\label{4.13}
		\hat{E}^{0}_{t+1}
		=\vec{A}^{t+1}\mathbb{P}_{0}&\rightarrow\frac{a+b+c+d}{4}
		\left(                 
		\begin{array}{llllllllllllllllllll}   
			1& 1  & \dots 
			& 1& 1 
		\end{array}
		\right)\mathbb{P}_{0}  =\frac{a+b+c+d}{4}=\frac{A}{4}.
	\end{align}
	Before proceeding further, let us recall that the expansion of $\hat{E}_{t+1}$ in equation~\eqref{5.1.8}~for general system parameter $K_{\lambda(t)}$
	\begin{align*}
		\hat{E}_{t+1}
		=&\left(                 
		\begin{array}{llll}   
			a & b & c & d   
		\end{array}
		\right)\Theta_{K_{\lambda(t)}}\mathbb{P}_{t}.             
	\end{align*}	
	Let
	\begin{align*}
		\epsilon(K_{\lambda(t)}):= \Theta_{K_{\lambda(t)}}-
		\Phi^{0},
	\end{align*}
	represents the error matrix between $\Theta_{K_{\lambda(t)}}$ and $\Phi^{0}$. So
	\begin{align}
		\hat{E}_{t+1}-\hat{E}^{0}_{t+1}
		=&\left(                 
		\begin{array}{llll}   
			a & b & c & d   
		\end{array}
		\right)\Theta_{K_{\lambda(t)}}\mathbb{P}_{t}             
		-\left(                 
		\begin{array}{llll}
			a & b & c & d   
		\end{array}
		\right)\Phi_{0}\mathbb{P}_{t}=\left(                 
		\begin{array}{llll}   
			a & b & c & d   
		\end{array}
		\right)\epsilon(K_{\lambda(t)})\mathbb{P}_{t}. 
	\end{align}
	According to Lemma \ref{lem:5.2}, we know that for any index $i,j,$
	$\left| \epsilon_{ij}(K_{\lambda(t)})\right|\leq MK_{\lambda(t)}$,
	where $M$ is a constant independent to the system parameters. The 4$\times$64 error matrix $\epsilon(K_{\lambda(t)})$ can be split into sixteen 4$\times$4 blocks, and compose a block diagonal matrix as follows 
	\begin{align*}
		\left(                 
		\begin{array}{llll}   
			a & b & c & d   
		\end{array}
		\right)\epsilon(K_{\lambda(t)})             
		&=\left(                 
		\begin{array}{lllllllll}   
			a & b & c & d &\dots & a & b & c & d  
		\end{array}
		\right)
		\left(                 
		\begin{array}{llll}   
			\epsilon_{1}(K_{\lambda(t)})  &  &  & \\ 
			&  & \ddots                     &\\
			&  &  & \epsilon_{16}(K_{\lambda(t)})
		\end{array}
		\right)\\
		&:=\left(                 
		\begin{array}{lllllllll}   
			a & b & c & d &\dots & a & b & c & d  
		\end{array}
		\right)\hat{\epsilon}(K_{\lambda(t)}).
	\end{align*}
	By Cauchy-Schwarz inequality
	\begin{align*}
		\left( \hat{E}_{t+1}-\hat{E}^{0}_{t+1}\right)^{2} 
		=&\left(\left(                 
		\begin{array}{llll}   
			a & b & c & d   
		\end{array}
		\right)\epsilon(K_{\lambda(t)})\mathbb{P}_{t}\right)^{2} =Tr^{2}\{\left(                 
		\begin{array}{lllllllll}   
			a & b & c & d &\dots & a & b & c & d  
		\end{array}
		\right)\hat{\epsilon}(K_{\lambda(t)})\mathbb{P}_{t}\}\\
		=&Tr^{2}\{\mathbb{P}_{t}\left(                 
		\begin{array}{lllllllll}   
			a & b & c & d &\dots & a & b & c & d  
		\end{array}
		\right)\hat{\epsilon}(K_{\lambda(t)})\}\\
		\leq & Tr\{\left(                 
		\begin{array}{lllllllll}   
			a & b &\dots &  c & d  
		\end{array}
		\right)^{T}\mathbb{P}^{T}_{t}\mathbb{P}_{t}\left(                 
		\begin{array}{lllllllll}   
			a & b &\dots & c & d  
		\end{array}
		\right)\}Tr\{\hat{\epsilon}^{T}(K_{\lambda(t)})\hat{\epsilon}(K_{\lambda(t)})\}\\
		\leq & C_{fg}Tr\{\hat{\epsilon}^{T}(K_{\lambda(t)})\hat{\epsilon}(K_{\lambda(t)})\}=C_{fg}\left\| \epsilon(K_{\lambda(t)})\right\|^{2}_{F}			\leq\tilde{C}_{fg}K^{2}_{\lambda(t)},
	\end{align*}
	where $C_{fg}$ and $\tilde{C}_{fg}$ depend only on constants $a,b,c,d$, that is to say, depends only on $f,g$. In view of the result in equation~\eqref{4.13}~, as $t \rightarrow \infty$, $ \hat{E}^{0}_{t+1}\rightarrow\frac{A}{4}
	$. Thanks to
	\begin{align*}
		\left(\hat{E}_{t+1}-\hat{E}^{0}_{t+1}\right)^{2} \leq \tilde{C}_{fg}K^{2}_{\lambda(t)},
	\end{align*}
	we conclude that the observable $\hat{\mathcal{S}}^{fg}_{\lambda(t)}(x_{1}, x_{2})$ tends to $\frac{A}{4}$ for any system parameter satisfies $ K_{\lambda(t)}\rightarrow 0$  as $t \rightarrow \infty$.
	Noticed that the structure of~\eqref{5.1.6}~and~\eqref{5.1.7}~ are similar, the estimate above also works if we substitute parameters $a,b,c,d$ by $\tilde{a},\tilde{b},\tilde{c},\tilde{d}$, which means as $t \rightarrow \infty$ and $ K_{\lambda(t)}\rightarrow 0$,
	$$\|fg\|^{2}_{L^{2}(K_{\lambda(t)})}\rightarrow \frac{C^{4}_{12}C^{2}_{12}}{4}
	, \tilde{\mathcal{S}}^{fg}_{\lambda(t)}(x_{1}, x_{2})\rightarrow\frac{C^{4}_{12}C^{2}_{12}}{4}\log\left(\frac{C^{4}_{12}C^{2}_{12}}{4}\right).$$		Finally		$$\mathcal{S}^{fg}_{\lambda(t)}(x_{1}, x_{2})\rightarrow\frac{A}{4}-\frac{C^{4}_{12}C^{2}_{12}}{4}\log \left(\frac{C^{4}_{12}C^{2}_{12}}{4}\right).$$
	This completes the estimate of first term of two point observables.
	
	For the second term $\mathcal{S}^{fg}(\lambda(t) x_{1}, \lambda(t) x_{2})$, recall that the initial distribution 
	$\mathbb{P}_{0}$ satisfies equation~\eqref{eq:5.2}~. According to the result in Theorem \ref{thm:3.2}, for any fixed $x_{1},x_{2}$ and $\lambda(t)$, taking $\Lambda(t)=\{1,2,\cdots,N(t)\}$ large enough such that $\lambda(t)x_{1},\lambda(t)x_{2}\in\Lambda(t)$. To be convenience, denote that
	\begin{align*}
		\mathbb{P}_{0}^{\Lambda(t)}&:=\mathbb{P}_{0}(\phi(x,0)=\sigma_{x},x\in \Lambda(t)|\phi(y,0)=\sigma_{y},y\in \Lambda^{c}(t)),\mathbb{P}_{0}^{\Lambda^{c}(t)}:=\mathbb{P}_{0}(\phi(x,0)=\sigma_{x},x\in \Lambda^{c}(t)).
	\end{align*}
	Hence,
	\begin{align*}
		\hat{\mathcal{S}}^{fg}_{N(t)}(\lambda(t) x_{1}, \lambda(t) x_{2})&=\sum\limits_{\substack{\sigma_{x}=\pm1 \\ x\in\Lambda^{c}(t)}}\hat{\mathcal{S}}_{N(t)}^{fg,\partial_{\Lambda(t)}}(\lambda(t) x_{1}, \lambda(t) x_{2})\mathbb{P}_{0}^{\Lambda^{c}(t)}=\sum_{k=1}^{4}\hat{\mathcal{S}}_{N(t)}^{fg,\partial_{\Lambda_{k}(t)}}(\lambda(t) x_{1}, \lambda(t) x_{2})\mathbb{P}_{0}(\partial_{\Lambda_{k}(t)}).		\end{align*}
	Let $i=\lambda(t) x_{1},j=\lambda(t) x_{2}$ and $\Lambda(t)=\{1,2,\cdots,N(t)\},\partial_{\Lambda}=\{0,N(t)+1\}$, we can get
	\begin{align*}	\hat{\mathcal{S}}_{N(t)}^{fg}(\lambda(t) x_{1}, \lambda(t) x_{2})&=\sum_{k=1}^{4}\hat{\mathcal{S}}_{N(t)}^{fg,\partial_{\Lambda_{k}(t)}}(\lambda(t) x_{1}, \lambda(t) x_{2})\mathbb{P}_{0}(\partial_{\Lambda_{k}(t)})=\sum_{k=1}^{4}(I_{1}^{k}+I_{2}^{k})\mathbb{P}_{0}(\partial_{\Lambda_{k}(t)}),\\
		\tilde{\mathcal{S}}_{N(t)}^{fg}(\lambda(t) x_{1}, \lambda(t) x_{2})&=\sum_{k=1}^{4}\tilde{\mathcal{S}}_{N(t)}^{fg,\partial_{\Lambda_{k}(t)}}(\lambda(t) x_{1}, \lambda(t) x_{2})\mathbb{P}_{0}(\partial_{\Lambda_{k}(t)})=\sum_{k=1}^{4}J_{k}\log J_{k}\mathbb{P}_{0}(\partial_{\Lambda_{k}(t)}).
	\end{align*}
	Notice that for any boundary condition and  time $t$ 
	$$\sum_{k=1}^{4}\mathbb{P}_{0}(\partial_{\Lambda_{k}(t)})=1.$$
	According to equation~\eqref{eq:3.13}~, taking limits $t\rightarrow+\infty$ and $N(t)\rightarrow+\infty$,
	\begin{align*}	
		\lim\limits_{t\rightarrow+\infty}\hat{\mathcal{S}}_{N(t)}^{fg}(\lambda(t) x_{1}, \lambda(t) x_{2})&=\lim\limits_{t\rightarrow+\infty}(M_{11}-M_{21})\lambda_{+}^{-j-1}\mathbb{P}_{0}(\partial_{\Lambda_{1}})+(M_{21}+M_{11})\lambda_{+}^{-j-1}\mathbb{P}_{0}(\partial_{\Lambda_{2}})\\
		&+\lim\limits_{t\rightarrow+\infty}(M_{11}-M_{21})\lambda_{+}^{-j-1}\mathbb{P}_{0}(\partial_{\Lambda_{3}})+(M_{21}+M_{11})\lambda_{+}^{-j-1} \mathbb{P}_{0}(\partial_{\Lambda_{4}})\\ &+\lim\limits_{t\rightarrow+\infty}(L_{11}-L_{21})\lambda_{+}^{-j-1}\mathbb{P}_{0}(\partial_{\Lambda_{1}})+(L_{21}+L_{11})\lambda_{+}^{-j-1}\mathbb{P}_{0}(\partial_{\Lambda_{2}})\\
		&+\lim\limits_{t\rightarrow+\infty}(L_{11}-L_{21})\lambda_{+}^{-j-1}\mathbb{P}_{0}(\partial_{\Lambda_{3}})+(L_{21}+L_{11})\lambda_{+}^{-j-1} \mathbb{P}_{0}(\partial_{\Lambda_{4}}).
	\end{align*}
	However
	\begin{align*}
		\lim\limits_{t\rightarrow+\infty}M_{21}\lambda_{+}^{-j-1}&=\lim\limits_{t\rightarrow+\infty}\hat{C}^{1}_{21}\hat{C}^{2}_{11}\lambda_{+}^{-i-3}\lambda_{-}^{i-1}+\hat{C}^{1}_{22}\hat{C}^{2}_{21}\lambda_{-}^{j-3}\lambda_{+}^{-j-1}\\
		&=\lim\limits_{t\rightarrow+\infty}\hat{C}^{1}_{21}\hat{C}^{2}_{11}\lambda_{+}^{-4}(\frac{\lambda_{-}}{\lambda_{+}})^{\lambda(t) x_{1}-1}+\hat{C}^{1}_{22}\hat{C}^{2}_{21}\lambda_{-}^{-4}(\frac{\lambda_{-}}{\lambda_{+}})^{\lambda(t) x_{2}+1}=0,\\
		\lim\limits_{t\rightarrow+\infty}L_{21}\lambda_{+}^{-j-1}&=\lim\limits_{t\rightarrow+\infty}\hat{C}^{4}_{21}\hat{C}^{3}_{11}\lambda_{+}^{-i-3}\lambda_{-}^{i-1}+\hat{C}^{4}_{22}\hat{C}^{3}_{21}\lambda_{-}^{j-3}\lambda_{+}^{-j-1}\\
		&=\lim\limits_{t\rightarrow+\infty}\hat{C}^{4}_{21}\hat{C}^{3}_{11}\lambda_{+}^{-4}(\frac{\lambda_{-}}{\lambda_{+}})^{\lambda(t) x_{1}-1}+\hat{C}^{4}_{22}\hat{C}^{3}_{21}\lambda_{-}^{-4}(\frac{\lambda_{-}}{\lambda_{+}})^{\lambda(t) x_{2}+1}=0.
	\end{align*}   
	Then
	\begin{align*}	
		\lim\limits_{t\rightarrow+\infty}\hat{\mathcal{S}}_{N(t)}^{fg}(\lambda(t) x_{1}, \lambda(t) x_{2})&=\lim\limits_{t\rightarrow+\infty}M_{11}\lambda_{+}^{-j-1}\left(\mathbb{P}_{0}(\partial_{\Lambda_{1}})+\mathbb{P}_{0}(\partial_{\Lambda_{2}})+\mathbb{P}_{0}(\partial_{\Lambda_{3}})+\mathbb{P}_{0}(\partial_{\Lambda_{4}})\right)\\
		&+\lim\limits_{t\rightarrow+\infty}L_{11}\lambda_{+}^{-j-1}\left(\mathbb{P}_{0}(\partial_{\Lambda_{1}})+\mathbb{P}_{0}(\partial_{\Lambda_{2}})+\mathbb{P}_{0}(\partial_{\Lambda_{3}})+\mathbb{P}_{0}(\partial_{\Lambda_{4}})\right)\\
		&=\lim\limits_{t\rightarrow+\infty}\lambda_{+}^{-j-1}\left((\hat{C}^{1}_{11}\hat{C}^{2}_{11}+\hat{C}^{4}_{11}\hat{C}^{3}_{11})\lambda_{+}^{j-3}+(\hat{C}^{1}_{12}\hat{C}^{2}_{21}+\hat{C}^{4}_{12}\hat{C}^{3}_{21})\lambda_{+}^{i-1}\lambda_{-}^{j-i-2}\right)\\
		&=(\hat{C}^{1}_{11}\hat{C}^{2}_{11}++\hat{C}^{4}_{11}\hat{C}^{3}_{11})\lambda_{+}^{-4}=\frac{1}{4}C_{12}^{1}C_{12}^{2}+\frac{1}{4}C_{12}^{4}C_{12}^{3}=\frac{A}{4}.
	\end{align*}
	Similarly
	\begin{align*}	
		\lim\limits_{t\rightarrow+\infty}\|fg\|^{2}_{L^{2}(\{\sigma\},\mathbb{P}_{0})}&=\lim\limits_{t\rightarrow+\infty}R_{11}\lambda_{+}^{-j-1}\left(\mathbb{P}_{0}(\partial_{\Lambda_{1}})+\mathbb{P}_{0}(\partial_{\Lambda_{2}})+\mathbb{P}_{0}(\partial_{\Lambda_{3}})+\mathbb{P}_{0}(\partial_{\Lambda_{4}})\right)\\
		&=\lim\limits_{t\rightarrow+\infty}\lambda_{+}^{-j-1}\left(\hat{C}^{4}_{11}\hat{C}^{2}_{11}\lambda_{+}^{j-3}+\hat{C}^{4}_{12}\hat{C}^{2}_{21}\lambda_{+}^{i-1}\lambda_{-}^{j-i-2}\right)=\hat{C}^{4}_{11}\hat{C}^{2}_{11}\lambda_{+}^{-4}=\frac{1}{4}C_{12}^{4}C_{12}^{2}.
	\end{align*}
	Hence
	\begin{align*}
		\lim_{t \to \infty}\tilde{\mathcal{S}}_{N(t)}^{fg}(\lambda(t) x_{1}, \lambda(t) x_{2})
		&=\frac{1}{4}C_{12}^{4}C_{12}^{2}\log\left(\frac{1}{4}C_{12}^{4}C_{12}^{2}\right).
	\end{align*}
	This completes the proof.
\end{proof}
	\begin{corollary}
		Taking the factor $Z(\lambda(t))=1$ in renormalization group equation. Then the new observable renormalization group equation can be written as
		\begin{align*}
			\mathcal{S}^{fg}_{\lambda(t)}(x_{1}, x_{2})-
			\mathcal{S}^{fg}(\lambda(t) x_{1}, \lambda(t) x_{2})&\rightarrow0,
		\end{align*}
as $t\rightarrow \infty$ and $\lambda(t) \rightarrow \infty$.
	\end{corollary}

\section{The estimate of any finite-point observables}\label{sec:6}
In this section, the calculation of two point observables will be extended to arbitrary finite point, which enormously  enriches the observables researchers interested in. Using the symbols introduced above, consider the observable $f(\phi(x_{1}),\cdots\,\phi(x_{m}))\in L^{2}(\mathbb{R}^{m})$. Let $$\{f^{2}(\phi(x_{1}),\cdots\,\phi(x_{m})):\phi(x_{i})=\sigma_{i} =\pm1\}:=\{a_{1},a_{2},\dots ,a_{2^{m}}\}.$$
Define the $n$-point observable  $\mathcal{E}^{f}_{\lambda(t)}( x_{1},\cdots,x_{m} )$ as follows 
\begin{align*}		\mathcal{E}^{f}_{\lambda(t)}( x_{1},\cdots,x_{m} )=&\mathbb{E}_{K_{\lambda(t)}}\left[f^{2}(\phi(x_{1},t+1),\cdots,\phi(x_{m},t+1))) \right],
\end{align*}
which can be expressed like equation~\eqref{5.1.6}~ 
\begin{align}\label{6.3}
	\mathcal{E}^{f}_{\lambda(t)}&( x_{1},\cdots,x_{m} )\nonumber\\   =&a_{1}\mathbb{P}\left(\phi(x_{1},t+1)=1,\cdots,\phi(x_{m},t+1)=1\right)+a_{2}\mathbb{P}\left(\phi(x_{1},t+1)=-1,\cdots,\phi(x_{m},t+1)=1\right)\nonumber\\
	&	\vdots\\
	+&a_{2^{m}-1}\mathbb{P}\left(\phi(x_{1},t+1)=1,\cdots,\phi(x_{m},t+1)=-1\right)+a_{2^{m}}\mathbb{P}\left(\phi(x_{1},t+1)=-1,\cdots,\phi(x_{m},t+1)=-1\right).\nonumber
\end{align} 
To analyse the properties of $\mathcal{E}^{f}_{\lambda(t)}( x_{1},\cdots,x_{m} )$, we need to extend  Lemma \ref{lem:5.1} firstly.
\begin{lemma}\label{lem:6.1}
	Assume that $\hat{\Phi}$ is a real symmetric matrix introduced in Theorem \ref{thm:6.2}  with  $\hat{\Phi}_{ij}=\hat{\Phi}_{ji}>0$
	\begin{align*}
		\hat{\Phi}=
		\left(                 
		\begin{array}{lll}
			\dots & \prod\limits_{k=1}^{m}\Phi((1-\gamma)\sigma_{3k-1}) &  \dots \\
			\dots & \Phi(-(1-\gamma)\sigma_{2})\prod\limits_{k=2}^{m}\Phi((1-\gamma)\sigma_{3k-1}) &  \dots \\
			& \qquad \vdots\\
			\dots & \prod\limits_{k=1}^{m}\Phi(-(1-\gamma)\sigma_{3k-1}) &  \dots 
		\end{array}
		\right),
	\end{align*}
	where $\sigma_{i}=\pm 1 $ are spin variables.  The signal $+$ and $-$ before spins $\{\sigma_{3k-1}: k=1,\dots,m\}$ in the $i$-th row of $\hat{\Phi}$ are determined by the binary sort in $i$-th row of  $2^{m}$ vector $\mathbb{A}_{t+1}$ defined in~\eqref{6.5}~, and the values of spins $\{\sigma_{3k-1}: k=1,\dots,m\}$ in $j$-th row of $\hat{\Phi}$ are determined by the  binary sort in  $j$-th row of  $\mathbb{A}_{t+1}$. Then we claim that $\hat{\Phi}$  also  admits the results in Lemma \ref{lem:5.1}, which means
	\begin{align*}
		\xi_{1}^{T}
		=\frac{1}{\sqrt{2^{m}}}\left(                 
		\begin{array}{cccc}   
			1&1&\cdots&1
		\end{array}
		\right), \quad and \quad 
		\hat{\Phi}^{t+1}\rightarrow\xi_{1}\xi_{1}^{T}
		=\frac{1}{2^{m}}\mathds{1}_{2^{m}}  \ as \ t\rightarrow \infty.
	\end{align*} 
\end{lemma}	
\begin{proof}
	Summing over the columns, we have $\sum_{k=1}^{2^{m}}\hat{\Phi}_{kj}=(\Phi((1-\gamma))+\Phi(-(1-\gamma))^{m}=1,\forall j=1,\cdots, 2^{m}$
	Obviously, 1 is a eigenvalue of $\hat{\Phi}$ with eigenvector $\xi_{1}$ and $\xi_{1}^{T}
	=\frac{1}{\sqrt{2^{m}}}\left(1\ 1 \cdots\ 1         \right) $.
	By Gersgorin Circle Theorem, each eigenvalue $\lambda$ of matrix  $\hat{\Phi}$ satisfies$	|\lambda|\leq |\lambda-\hat{\Phi}_{ii}|+\hat{\Phi}_{ii}\leq \sum_{k=1,k \neq i}^{2^{m}}\hat{\Phi}_{ki}+\hat{\Phi}_{ii}=1$.		
	By Gersgorin Circle Theorem, the spectral radius of  $\hat{\Phi}$ is less than or equal to 1. The key point is to prove that the algebraic multiplicity of the largest eigenvalue 1 is single, which is equivalent to $Rank(\hat{\Phi}-I)=2^{m}-1$. Noticed that 
	\begin{align}\label{6.42}
		Rank(\hat{\Phi}-I)
		=2^{m}-1
		&\Longleftrightarrow Rank\left(                 
		\begin{array}{lll}
			\hat{\Phi}_{11}-1 & \dots  & \hat{\Phi}_{12^{m}}\\
			\vdots & \ddots  &\vdots \\
			\hat{\Phi}_{2^{m}1} & \dots  &  \hat{\Phi}_{2^{m}2^{m}}-1 
		\end{array}
		\right)=2^{m}-1\nonumber\\
		&\Longleftarrow Rank\left(                 
		\begin{array}{llll}
			\hat{\Phi}_{22}-1  &\dots& \hat{\Phi}_{22^{m}}\\
			\vdots &  \ddots &\vdots \\
			\hat{\Phi}_{2^{m}2}  & \dots  & \hat{\Phi}_{2^{m}2^{m}}-1 
		\end{array}
		\right)=2^{m}-1.
	\end{align} 
	Due to any $\hat{\Phi}_{ij}>0$, for submatrix 
	\begin{align*}
		\tilde{\Phi}=\left(                 
		\begin{array}{llll}
			\hat{\Phi}_{22}  &\dots& \hat{\Phi}_{22^{m}}\\
			\vdots &  \ddots &\vdots \\
			\hat{\Phi}_{2^{m}2}  & \dots  & \hat{\Phi}_{2^{m}2^{m}}
		\end{array}
		\right) ,
	\end{align*} 	
	By Gersgorin Circle Theorem, for $ 1\leq i\leq2^{m}-1$, each eigenvalue $\tilde{\lambda}$ satisfies
	$|\tilde{\lambda}-\tilde{\Phi}_{ii}|\leq \sum_{k=1,k \neq i}^{2^{m}-1}\hat{\Phi}_{ki}$. Then
	\begin{align*}
		|\tilde{\lambda}|\leq |\tilde{\lambda}-\tilde{\Phi}_{ii}|+\tilde{\Phi}_{ii}\leq \sum_{k=1,k \neq i}^{2^{m}-1}\hat{\Phi}_{ki}+\hat{\Phi}_{ii}=1-\hat{\Phi}_{1i}<1.	
	\end{align*} 		
	the spectral radius of the submatrix $\tilde{\Phi}$ is smaller than 1 strictly, which means 1 is not the eigenvalue of $\tilde{\Phi}$, then $\tilde{\Phi}-I$ is invertible, which leads to $Rank(\tilde{\Phi}-I)=2^{m}-1$. Thanks to the equivalence in~\eqref{6.42}~, $Rank(\hat{\Phi}-I)=2^{m}-1$. The rest of argument is trivial.
\end{proof}
\begin{theorem}\label{thm:6.2}
	Considering the stochastic dynamic~\eqref{5.11}~ with system parameter $K_{\lambda(t)}$ and initial distribution  $\mathbb{P}_{0}$. Choosing a  scaling parameter $\lambda(t)\uparrow +\infty$ as  $t\rightarrow +\infty$, and the parameter $K_{\lambda(t)}$ satisfies  $K_{\lambda(t)}\rightarrow K^{*}$. Then for positive integer $m\geq 1$, the  finite point observable $\mathcal{E}^{f}_{\lambda(t)}( x_{1},\cdots,x_{m} )$ satisfies
	\begin{align*}
		\mathcal{E}^{f}_{\lambda(t)}(x_{1},\cdots,x_{m})\rightarrow\mathbb{E}_{K^{*}=0}\left[ f^{2}\right],as \ \ t\rightarrow \infty,
	\end{align*}
	where $\mathbb{E}_{K^{*}=0}\left[ f^{2}\right]$ represents the observable at fixed point and the results above are independent to $\mathbb{P}_{0}$.
\end{theorem}
\begin{proof}
	The main idea is same as Theorem \ref{th:5.3}. 
	Obviously, the probability distribution of $m$ dimensional spin variables $(S_{1},\cdots,S_{m})$ at time $t+1$ are totally determined by  $3m$ dimensional spin variables $(\sigma_{1},\cdots,\sigma_{3m})$ at time $t$. Undergoing the structure rearrangement of the $3m$ dimensional spin variables in a binary sort 
	\begin{align}\label{6.4}
		&\mathbb{P}_{t}=\left(                 
		\begin{array}{ccccccc}
			\\
			P^{1}_{t}\\
			P^{2}_{t}\\	
			P^{3}_{t}\\
			P^{4}_{t}\\
			\vdots \\
			P^{2^{3m-1}}_{t}\\
			\vdots\\
			P^{2^{3m}}_{t}
		\end{array}
		\right)\leftrightarrow            
		\left(                 
		\begin{array}{rrrrrrrr}
			\sigma_{1}&\sigma_{2}&\cdots&\cdots&\sigma_{3m-1}&\sigma_{3m}\\
			1 &1 &\cdots&\cdots&1 &1\\
			-1&1 &\cdots&\cdots&1 &1\\	
			1 &-1 &\cdots&\cdots&1 &1\\
			-1&-1 &\cdots&\cdots&1 &1\\
			\qquad &\vdots &\qquad &\vdots\\
			1 &1 &\cdots&\cdots&1 &-1\\
			\qquad &\vdots &\qquad &\vdots\\
			-1&-1 &\cdots&\cdots&-1 &-1
		\end{array}
		\right),            
	\end{align}	
	where 
	\begin{align*}	P^{1}_{t}=&\mathbb{P}(\phi(x_{1}-1,t)=1,\phi(x_{1},t)=1,\phi(x_{1}+1,t)=1,\\
		&  \cdots,  \phi(x_{m}-1,t)=1,\phi(x_{m},t)=1,\phi(x_{m}+1,t)=1),
	\end{align*}	
	and $\{P_{t}^{i}\}_{i=1,\cdots,2^{3m}}$ are the probability values corresponding to these  $3m$ dimensional spins $(\sigma_{1},\cdots,\sigma_{3m})$. The probability at time $t+1$ above can also be written as~\eqref{5.1.3}~similarly.  Define $\mathbb{A}_{t+1}$ as follows
	\begin{align}\label{6.5}
		\mathbb{A}_{t+1}  := \left(                 
		\begin{array}{c}
			\mathbb{P}\left(\phi(x_{1},t+1)= 1,\cdots, \phi(x_{m},t+1)= 1\right)  \\  
			\mathbb{P}\left(\phi(x_{1},t+1)=-1,\cdots, \phi(x_{m},t+1)= 1\right) \\  
			\vdots \\
			\mathbb{P}\left(\phi(x_{1},t+1)= 1,\cdots \phi(x_{m},t+1)=-1\right)\\
			\mathbb{P}\left(\phi(x_{1},t+1)=-1,\cdots, \phi(x_{m},t+1)=-1\right)
		\end{array}
		\right)           
		:= \Theta_{K_{\lambda(t)}}\mathbb{P}_{t},
	\end{align}	
	where $\Theta_{K_{\lambda(t)}}$ is the $2^{m}\times2^{3m}$ 
	transition probability matrix, which connects  the joint distribution $\mathbb{A}_{t+1}$ of $m$-point new spin variables at time $t+1$ to the joint distribution $\mathbb{P}_{t}$ of original $3m$-point spins at time $t$.
	
	In order to calculate more concretely, let us consider the case $K_{\lambda(t)}=0,t\geq0$ firstly.
	Denote that  $\Phi^{0}$ is the transition probability matrix at $K_{\lambda(t)}=0$. Similar to the argument in Theorem \ref{th:5.3}, the columns in $\Phi^{0}$ are totally determined by $\{\sigma_{3k-1},k=1,\cdots,m\}$, so there are only $2^{m}$ different columns in $\Phi^{0}$, which compose the so-called $2^m$ dimensional \emph{ tranfer determined matrix } $\hat{\Phi}$.  $\hat{\Phi}$ can be expressed as
	\begin{align*}
		\hat{\Phi}=
		\left(                 
		\begin{array}{lll}
			\dots & \prod\limits_{k=1}^{m}\Phi((1-\gamma)\sigma_{3k-1}) &  \dots \\
			\dots & \Phi(-(1-\gamma)\sigma_{2})\prod\limits_{k=2}^{m}\Phi((1-\gamma)\sigma_{3k-1}) &  \dots \\
			& \qquad \vdots\\
			\dots & \prod\limits_{k=1}^{m}\Phi(-(1-\gamma)\sigma_{3k-1}) &  \dots 
		\end{array}
		\right),
	\end{align*}
	where $\sigma=\pm 1$ are spins. The signal $+$ and $-$ before spins $\{\sigma_{3k-1}: k=1,\dots,m\}$ in the $i$-th row of $\hat{\Phi}$ are determined by the binary sort in $i$-th row of  $2^{m}$ vector $\mathbb{A}_{t+1}$, and the values of spins $\{\sigma_{3k-1}: k=1,\dots,m\}$ in $j$-th row of $\hat{\Phi}$ are determined by the  binary sort in  $j$-th row of  $\mathbb{A}_{t+1}$.Then, $\hat{\Phi}$ is a real symmetric matrix with elements $\hat{\Phi}_{ij}=\hat{\Phi}_{ji}>0$. The iteration similar to~\eqref{4.14}~in Theorem \ref{th:5.3} can be established
	\begin{align}\label{4.47}
		\hat{E}^{0}_{t+1}
		=&\left(                 
		\begin{array}{llllllllllllllllllll}   
			a_{1} & a_{1} & a_{2} & a_{2} & \dots & a_{2^m-1} & a_{2^m-1} & a_{2^m} & a_{2^m}  & a_{1} & a_{1}  & \dots &  a_{2^m} & a_{2^m}
		\end{array}
		\right)\mathbb{P}_{t+1}\nonumber\\
		=&\left(                 
		\begin{array}{llllllllllllllllllll}   
			a^{1}_{1} & a^{1}_{1} & a^{1}_{2} & a^{1}_{2} & \dots & a^{1}_{2^m-1} & a^{1}_{2^m-1} & a^{1}_{2^m} & a^{1}_{2^m}  & a^{1}_{1} & a^{1}_{1}  & \dots &  a^{1}_{2^m} & a^{1}_{2^m}
		\end{array}
		\right)\mathbb{P}_{t}\nonumber\\
		&\vdots \\
		=&\left(                 
		\begin{array}{llllllllllllllllllll}   
			a^{t+1}_{1} & a^{t+1}_{1} & a^{t+1}_{2} & a^{t+1}_{2} & \dots & \dots &  a^{t+1}_{2^m} & a^{t+1}_{2^m}
		\end{array}
		\right)\mathbb{P}_{0},\nonumber
	\end{align}	
	where the parameters in each step above are determined by
	\begin{align*}
		\left(                 
		\begin{array}{llll}   
			a_{1}^{t+1} & a_{2}^{t+1} & \dots & a_{2^{m}}^{t+1}    
		\end{array}
		\right)
		=&\left(                 
		\begin{array}{llll}   
			a_{1}^{t} &  \dots & a_{2^{m}}^{t}     
		\end{array}
		\right)\hat{\Phi}
		=\cdots 
		=\left(                 
		\begin{array}{llll}   
			a_{1} & a_{2} & \dots & a_{2^{m}}   
		\end{array}
		\right) \hat{\Phi}^{t+1}.
	\end{align*}
	According to Lemma \ref{lem:6.1}, as $t\rightarrow \infty$ 		\begin{align*}
		\left(                 
		\begin{array}{llll}   
			a_{1}^{t+1} & a_{2}^{t+1} & \dots & a_{2^{m}}^{t+1}   
		\end{array}
		\right)
		&\rightarrow\left(                 
		\begin{array}{llll}   
			a_{1} & a_{2} & \dots  & a_{2^{m}}   
		\end{array}
		\right)\xi_{1}\xi_{1}^{T}=\frac{\sum_{i=1}^{2^{m}}a_{i}}{2^{m}}\left(                 
		\begin{array}{llll}   
			1 & 1 & \dots  & 1   
		\end{array}
		\right).
	\end{align*}
	Because the initial data $\mathbb{P}_{0}$ is a probability distribution, then as $t\rightarrow \infty$, 
	\begin{align*}
		\hat{E}^{0}_{t+1}:=&\mathbb{E}_{K_{\lambda(t)}=0}\left[ f^{2}\right]=\left(                 
		\begin{array}{llllllllllllllllllll}   
			a^{t+1}_{1} & a^{t+1}_{1} & a^{t+1}_{2} & a^{t+1}_{2} & \dots & \dots &  a^{t+1}_{2^m} & a^{t+1}_{2^m}
		\end{array}
		\right)\mathbb{P}_{0}		\rightarrow\frac{\sum_{i=1}^{2^{m}}a_{i}}{2^{m}}.
	\end{align*}
	For general  $K_{\lambda(t)}$, 
	the error matrix between $\Theta_{K_{\lambda(t)}}$ and $\Phi^{0}$ can also be defined as $\epsilon(K_{\lambda(t)}):= \Theta_{K_{\lambda(t)}}-
	\Phi^{0}$. The rest of analysis is similar to Theorem \ref{th:5.3}, and we can get that the arbitrary finite point observable
	$\mathcal{E}^{f}_{\lambda(t)}( x_{1},\cdots,x_{m} )$ converges for  $\mathbb{P}_{0}$ and $K_{\lambda(t)}$ satisfies $ K_{\lambda(t)}\rightarrow 0$  as $t \rightarrow \infty$, which completes the proof.
\end{proof}

\begin{remark}
	These results indicate that once the scaling parameter satisfies
	$\lambda(t) \rightarrow \infty$, as $t\rightarrow \infty$
	and $K_{\lambda(t)}\rightarrow K^{*}$, as $\lambda(t)\rightarrow \infty.$
	The observable satisfies $\mathcal{E}^{f}_{\lambda(t)}(x_{1},\cdots,x_{m})\rightarrow\mathbb{E}_{K^{*}=0}\left[ f^{2}\right]$ as $t\rightarrow \infty$. For the two point correlation function
	\begin{align*}
		\langle \phi(x_{1},t),\phi(x_{2},t) \rangle_{\lambda(t)}	&=\mathbb{E}_{K_{\lambda(t)}}\left[ \phi(x_{1},t),\phi(x_{2},t)\right]-\mathbb{E}_{K_{\lambda(t)}}\left[ \phi(x_{1},t)\right]\mathbb{E}_{K_{\lambda(t)}}\left[ \phi(x_{2},t)\right].
	\end{align*}
	In the case $K^{*}=0$, the spin variable $\phi(x_{1},t)$ is independent to $\phi(x_{2},t)$ according to the definition. Hence, taking the limit $t\rightarrow \infty $, we know that
	\begin{align*}
		\langle \phi(x_{1},t),\phi(x_{2},t) \rangle_{\lambda(t)}	\rightarrow\langle \phi(x_{1},t),\phi(x_{2},t) \rangle_{K^{*}=0}=0,
	\end{align*}
	which coincides with the results in classical one dimensional Ising model.
	
\end{remark}

\section*{Acknowledgement}
This work is supported by the National Natural Science Foundation of China (Grant No. 12288201).



\end{document}